\documentclass[11pt]{article}

\usepackage{amsrefs}
\usepackage{amsmath,amssymb,amsthm}
\usepackage{enumerate,tikz}
\usepackage[all,cmtip]{xy}
\usepackage[applemac]{inputenc}
\usepackage[colorlinks=true, linkcolor=black, citecolor=black]{hyperref}

\newcommand \A{{\mathbb A}}
\newcommand \Ad{\operatorname{Ad}}
\newcommand \ad{\operatorname{ad}}
\newcommand \adm{\mathrm{adm}}
\newcommand \al{\alpha}
\newcommand \be{\beta}
\newcommand \bs{\backslash}
\newcommand \C{{\mathbb C}}
\newcommand \CA{\mathcal{A}}

\newcommand \CC{\mathcal{C}}
\newcommand \CD{\mathcal{D}}
\newcommand \CE{\mathcal{E}}

\newcommand \CG{\mathcal{G}}
\newcommand \CL{\mathcal{L}}
\newcommand \CM{\mathcal{M}}

\newcommand \CO{\mathcal{O}}
\newcommand \CP{\mathcal{P}}

\newcommand \diag{\operatorname{diag}}
\newcommand \dist{\operatorname{dist}}
\newcommand \eps{\varepsilon}
\newcommand \ev{\mathrm{ev}}
\newcommand \fin{\mathrm{fin}}
\newcommand \g{\mathfrak g}
\newcommand \Ga{\Gamma}
\newcommand \GL{\operatorname{GL}}
\newcommand \ga{\gamma}
\newcommand \geom{\mathrm{geom}}
\newcommand \h{\mathfrak h}
\newcommand \Hom{\operatorname{Hom}}

\newcommand \ind{\operatorname{ind}}

\newcommand \la{\lambda}
\newcommand \La{\Lambda}
\newcommand \Lie{\operatorname{Lie}}
\newcommand \M{\operatorname M}
\newcommand \m{\mathfrak m}
\newcommand \mqed{\tag*\qedhere}
\newcommand \n{\mathfrak n}
\newcommand \N{{\mathbb N}}
\newcommand \odd{\mathrm{odd}}
\newcommand \ol{\overline}
\newcommand \om{\omega}
\newcommand \Om{\Omega}
\newcommand \p{\mathfrak p}
\newcommand \Per{\operatorname{Per}}
\newcommand \PGL{\operatorname{PGL}}
\newcommand \PO{\operatorname{PO}}
\newcommand \Pol{\operatorname{Pol}}

\newcommand \Q{{\mathbb Q}}
\newcommand \R{{\mathbb R}}

\newcommand \Rep{\operatorname{Rep}}
\newcommand \res{\operatorname{res}}
\newcommand \SL{\operatorname{SL}}
\newcommand \sm{\smallsetminus}
\newcommand \SO{\operatorname{SO}}
\newcommand \spec{\mathrm{spec}}

\newcommand \st{\mathrm{st}}

\newcommand \tr{\operatorname{tr}}
\newcommand \triv{\mathrm{triv}}
\newcommand \ul{\underline}
\newcommand \vol{{\rm vol}}
\newcommand \what{\widehat}
\newcommand \Z{{\mathbb Z}}
\newcommand \z{\mathfrak z}

\renewcommand \1{{\bf 1}}
\renewcommand \a{\mathfrak a}

\renewcommand \H{\operatorname{H}}

\renewcommand \k{\mathfrak k}
\renewcommand \Re{\operatorname{Re}}
\renewcommand \({\left(}
\renewcommand \){\right)}

\newcommand{\e}
[1]{\emph{#1}\index{#1}}

\newcommand{\mat}
[4]{\(\begin{matrix}#1 & #2 \\ #3 & #4\end{matrix}\)}

\newcommand{\norm}
[1]{\left\|#1\right\|}
 
\newcommand{\smat}
[4]{\(\begin{smallmatrix}#1 & #2 \\ #3 & #4\end{smallmatrix}\)}

\newcommand{\bmat}
[4]{\begin{bmatrix}#1 & #2 \\ #3 & #4\end{bmatrix}}

\newcommand{\sbmat}
[4]{\left[\begin{smallmatrix}#1 & #2 \\ #3 & #4\end{smallmatrix}\right]}

\renewcommand{\sp}
[1]{\left\langle #1\right\rangle}

\newcommand{\tto}
[1]{\stackrel{#1}{\longrightarrow}}

\newtheorem{theorem}{Theorem}[section]
\newtheorem{conjecture}[theorem]{Conjecture}
\newtheorem{lemma}[theorem]{Lemma}
\newtheorem{corollary}[theorem]{Corollary}
\newtheorem{proposition}[theorem]{Proposition}

\theoremstyle{definition}
\newtheorem{definition}[theorem]{Definition}

\begin{document}

\pagestyle{myheadings} \markright{Lefschetz SL(3)}

\title{A prime Geodesic Theorem for $\SL_3(\Z)$}
\author{Anton Deitmar, Yasuro Gon \& Polyxeni Spilioti}
\date{}
\maketitle

{\bf Abstract:} We show a Prime Geodesic Theorem for the group $\SL_3(\Z)$ counting those geodesics whose lifts lie in the split Cartan subgroup.
This is the first arithmetic Prime Geodesic Theorem of higher rank for a non-cocompact group.

$$ $$

\tableofcontents

\newpage
\section*{Introduction}

In this paper we show a Prime Geodesic Theorem for congruence subgroups of $\SL_3(\Z)$.
This is the first example of a Prime Geodesic Theorem for a non-cocompact arithmetic group and geodesics in a Cartan subgroup  of split-rank bigger than 1.
The case of split-rank one has been considered in \cite{class}.
Higher rank cases in a cocompact setting have been considered in \cites{GAFA,classMP,primgeoMP,primgeoII} and for $p$-adic groups in \cite{primgeoBldg}.

We present the Prime Geodesic Theorem for $\SL_3(\Z)$ in a formulation involving class numbers as follows:
Let $O_\R(3)$ denote the set of all orders $\CO$ in totally real number fields $F$ of degree $3$. 
For such an order $\CO$ let $h(\CO)$ be
its class number,
$R(\CO)$ its regulator.
For $\la\in\CO^\times$ let $\rho_1,\rho_2,\rho_3$ denote the real embeddings of $F$ ordered
in a way that $|\rho_1(\la)|\ge |\rho_{2}(\la)|\ge |\rho_{3}(\la)|$. 
Let
$$
\al_1(\la)=\frac{|\rho_1(\la)\rho_3(\la)|}
{|\rho_{2}(\la)|^2},\quad \al_2(\la)=
\(\frac{|\rho_2(\la)|}
{|\rho_{3}(\la)|}\)^2
$$

{\bf Theorem.} (Prime Geodesic Theorem for $\SL_3(\Z)$)
For $T_1,\dots,T_r>0$ set
$$
\vartheta(T)=\sum_{\substack{\CO\in O_\R(3),\ \la\in\CO^\times/\pm
1\\
1<\al_1(\la)\le T_1\\
1<\al_2(\la)\le T_2\\ 
}}  R(\CO)\,
h(\CO),
$$
where $h(\CO)$ is the class number of $\CO$ and $R(\CO)$ its regulator.
Then we have, as $T_1,T_2\to \infty$,
$$
\vartheta(T_1,T_2)\ \sim\ \frac{16}{\sqrt{3}}\,T_1T_2.
$$
We further present a conjectural Lefschetz formula for general congruence arithmetic groups which, if proven, would imply a similar Prime Geodesic Theorem for arbitrary congruence arithmetic groups.

\section{Split groups}
\subsection*{Traces}
Let $G$ be a connected semisimple Lie group with finite center.
Throughout, we assume that $G$ is \e{split}  over the reals, i.e., there exists a split Cartan subgroup $H_{sp}$.
Write $\g$ for the Lie algebra of $G$ and $\g_\C$ for the complexification of $\g$. 
Let $K\subset G$ be a maximal compact subgroup with Cartan involution $\theta$ and let $P_{sp}=M_{sp} A_{sp} N_{sp} $ be a minimal parabolic subgroup such that $A_{sp} M_{sp} $ is $\theta$-stable.
Then $A_{sp} $ is $\theta$-stable and $M_{sp} $ is the centralizer of $A_{sp} $ in $K$.
Further, as $G$ is split, $M_{sp} $ actually is a finite group.

We will normalize all Haar measures of compact groups to volume 1 and all others according to Harish-Chandra's normalization as in \cite{HC-HA1}. Note that this normalizations depends on a choice of an invariant form $B$, i.e., a multiple of the Killing form, which we will consider fixed and only specify later.

We fix an irreducible representation $(\tau,V_\tau)$ of $K$.
Let $E_\tau$ denote the $G$-homogeneous vector bundle over $G/K$ induced by $\tau$.
The smooth sections of the bundle $E_\tau$ may be viewed as smooth functions $f:G\to V_\tau$ satisfying $f(xk)=\tau(k^{-1})f(x)$.
Fix some $0\ne\al\in V_\tau^*$, then $\al\circ f$ is a complex-valued function on $G$ of right $K$-type $\tau$ and the map $f\mapsto \al\circ f$ is a linear isomorphism
$$
\Ga^\infty(E_\tau)\tto\cong C^\infty(G)(*,\tau).
$$
Here we use the following notation: The compact group $K\times K$ acts on $C^\infty(G)$ by 
$$
(k,l)f(x)=f(k^{-1}xl).
$$
Accordingly, the space decomposes into $K$-bitypes
$$
C^\infty(G)=\ol{\bigoplus_{\ga,\tau\in\what K}}C^\infty(G)(\ga,\tau).
$$
We write $C^\infty(G)(*,\tau)$ for the closure of the sum of all $C^\infty(G)(\ga,\tau)$, $\ga\in\what K$.
The group $G$ acts on $C^\infty(G)$ by right translations $R(x)\phi(y)=\phi(yx)$.
For $f\in C_c^\infty(G)$ and any unitary representation $\pi$ of $G$ we write 
$$
\pi(f)=\int_Gf(x)\pi(x)\,dx,
$$
where $dx$ denotes a fixed Haar measure on $G$, which we will normalize later.
If $f$ is of left $K$-type $\tau$, then the operator $R(f)$ preserves the right $K$-type $C^\infty(G)(*,\tau)$.
The restriction of $R(f)$ to this right-$K$-type equals $R(P(f))$ where $P$ is the projection to $C^\infty(G)(*,\tau^*)$, where $\tau^*$ is the representation dual to $\tau$.
Together this means that we may assume $f\in C^\infty_c(G)(\tau,\tau^*)$.

\begin{definition}
Let $\what G$ denote the \e{unitary dual} of $G$, i.e., the set of all irreducible unitary representations up to unitary equivalence.

Let $\what G_\adm$ denote the set of all irreducible admissible representations up to infinitesimal equivalence. By results of Casselman and Harish-Chandra, every $\pi\in \what G_\adm$ can be realized on a Banach space and $\what G$ can be viewed as a subset of $\what G_\adm$.
\end{definition}

\begin{lemma}
Consider the convolution algebra $C^\infty_c(G)(\tau,\tau^*)$. and let $C_\tau$ denote its center.
Then $C_\tau$ equals the set of all $f\in C_c^\infty(G)$ such that
$$
\pi(f)=h_f(\pi)P_{\pi,\tau}
$$
holds for every $\pi\in\what G$, where $h_f(\pi)$ is a scalar and $P_{\pi,\tau}$ is the orthogonal projection onto the $K$-isotype $V_\pi(\tau)$.

If $f\in C_\tau$, the $\pi(f)=h_f(\pi)P_{\pi,\tau}$ also holds for every $\pi\in\what G_\adm$.
\end{lemma}

\begin{proof}
Let $\pi\in\what G$. Then $\pi(C_c^\infty(C)(\tau,\tau^*))=P_{\pi,\tau}\,\pi(C_c^\infty(G))\,P_{\pi,\tau}$.
Therefore, the action of $C_c^\infty(G)(\tau,\tau^*)$ on $V_\pi(\tau)$ is irreducible, hence its center acts by scalars by the Lemma of Schur.
The other way round, let $f\in C_c^\infty(G)(\tau,\tau^*)$ be such that $\pi(f)=h_f(\pi)P_{\pi,\tau}$ holds for every $\pi\in\what G$.
Then $\pi(f)$ commutes with $\pi(h)$ for every $h\in C_c^\infty(G)(\tau,\tau^*)$, hence $f$ lies in the center of $C_c^\infty(G)(\tau,\tau^*)$ by the Plancherel Theorem.
\end{proof}

For a principal series representation $\pi=\pi_{\sigma,\la}$ with $\sigma\in \what M_{sp} $ and $\la\in\a_{sp,\C}^*$ we can compute the scalar $h_f(\pi)=h_f(\sigma,\la)$ as follows.
First we assume that the $K$-type $\tau$ actually occurs in $\pi$, for otherwise the scalar will  be  zero.
Pick a norm one vector $p_0$ in 
$V_{\sigma,\la}(\tau)$, then $p_0$ is a continuous function with $p_0(1)\ne 0$ and, using the Peter-Weyl Theorem, one has $\int_Kf(xk)p_0(yk)\,dk=f(x)p_0(y)$.
With the Iwasawa integral formula one computes
\begin{align*}
h_f(\sigma,\la)p_0(1)&=\pi(f)p_0(1)\\
&=\int_G f(x)\,p_0(x)\,dx\\
&=\int_{A_{sp} N_{sp} K} f(ank)\,a^{\la+\rho}\,p_0(k)\,da\,dn\,dk\\
&=\int_{A_{sp} }f^{N_{sp} }(a)\,a^{\la+\rho}\,da\ p_0(1),
\end{align*}
where $f^{N_{sp} }(a)=\int_{N_{sp} }f(an)\,dn$.
We conclude that the scalar
$$
h_f(\sigma,\la)=\int_{A_{sp} }f^{N_{sp} }(a)\,a^{\la+\rho}\,da
$$
is independent of $\sigma$. We write it as $h_f(\la)$.
By the theory of Knapp-Stein intertwining operators we know that $\pi_{\sigma,\la}\cong\pi_{w\sigma,w\la}$ for $w\in W=W(G,A_{sp})$, so that
$$
h_f(\la)=h_f(w\la),\quad w\in W.
$$

\begin{proposition}\label{prop1.2}
The map $\Phi: f\mapsto f^{N_{sp} }(a)a^\rho$ is an isomorphism of topological algebras  $C_\tau\tto\cong C_c^\infty(A_{sp} )^W$.
\end{proposition}

\begin{proof}
Note first that since $h_{f*g}=h_fh_g$, the map $\Phi$ is a homomorphism of convolution algebras.
The map is linear and continuous. We only need to show bijectivity, as the open mapping theorem then implies continuity of the inverse.
For injectivity let $f\in\ker(\Phi)$.
By the above computation we then have $\pi(f)=0$ for every $\pi\in\what G$, so $f=0$ by the Plancherel Theorem.
For surjectivity we employ the Paley-Wiener Theorem for reductive groups by James Arthur \cite{ArthurPW}, see also \cites{vdBSchlicht,VdBSouf}.
In this paper, the Fourier transform of $K$-finite functions in $C_c^\infty(G)$ is described in terms of growth estimates and additional relations, the so-called \e{Arthur-Campoli relations}.
The Weyl group invariance is a special case of those.
We have to show that all Arthur-Campoli relations are satisfied by any function $h$ which is an $A_{sp} $-Fourier transform of some $g\in C_c^\infty(A_{sp} )^W$.
The Arthur-Campoli relations are of the following form: there are tuples of differential operators $D_1,\dots,D_n$ with constant coefficients on $\a_{sp,\C}^*$ and $\la_1,\dots,\la_n\in\a_{sp,\C}^*$, such that 
$$
\sum_{k=1}^n D_kh_f(\la_k)=0
$$
holds for every $f\in C_\tau$.
Here the differential operator $D_k$ acts with respect to $\la_k$.
We claim that if such a relation holds for all $h_f$ with $f\in C_\tau$, then it  holds for every $h\in C^\infty(\a_{sp,\C}^*)^W$.
Let $p\in\Pol(\a_{sp,\C}^*)^W$ be a Weyl-group invariant polynomial and let $z_p$ be the corresponding differential operator in $\z$ given by the Harish-Chandra isomorphism. Then one has $h_{z_pf}=ph_f$, so that $\sum_{k=1}^nD_k(p(\la_k)h_f(\la_k))=0$.
The Casimir-Element $C\in U(\g)$ induces an differential operator of order 2 on $E_\tau$. For an even Paley-Wiener function $\phi$ on $\C$ the functional calculus $\phi(\sqrt{-C})$ gives a smoothing operator of finite propagation speed \cite{CGT}, which by $G$-invariance is given by convolution with some $f\in C_\tau$.
Varying $\phi$ such that its support shrinks to zero, the corresponding eigenvalue functions $h_f$ tend to a non-zero constant locally uniformly with all derivatives. Hence it follows that we get $\sum_{k=1}^nD_k(p(\la_k))=0$ for every Weyl group invariant polynomial $p$.
Another approximation shows that $\sum_{k=1}^nD_k(h(\la_k))=0$ for every $h\in C^\infty(\a_{sp,\C}^*)^W$ as claimed.
\end{proof}

\begin{definition}
Let $B$ denote a fixed positive multiple of the Killing form on $\g$ and $\theta$ the Cartan involution fixing $K$ pointwise.
The form $\sp{X,Y}=-B(\theta(X),Y)$ is positive definite on $\g$ and induces a $G$-invariant Riemannian metric on $G/K$.
Let $\dist(x,y)$ denote the corresponding distance function and let $d(g)=\dist(gK,eK)$ for $g\in G$.

Let $U(\g_\C)$ denote the universal enveloping algebra of $\g_\C$.
Every element $X$ of $U(\g_\C)$ gives rise to a left-invariant differential operator, written $h\mapsto h*X$, and a right-invariant differential operator, written $h\mapsto X*h$.
Recall that for $p>0$ the $L^p$-Schwartz space $\CC^p(G)$ is defined as the space of all $h\in C^\infty(G)$ such that, for every $n\in\N$ and $X,Y\in U(\g_\C)$ the seminorm
$$
|h|_{p,n,X,Y}=\sup_{g\in G}|X*h*Y(g)|\ \Xi(g)^{-2/p}(1+d(g))^n
$$
is finite. Here $\Xi$ is the basic spherical function,
$$
\Xi(x)=\int_K \underline a(kx)^\rho\,dk,
$$
where $\underline a(x)$ is the $A_{sp}$-part of the $A_{sp} N_{sp} K$-decomposition of $x\in G$.
It suffices for our purposes to know that there are $r_1>r_2>0$ such that $e^{-r_1d(g)}\le \Xi(g)\le e^{-r_2d(g)}$.
If we complete the space $\CC^p(G)$ with respect to the seminorms involving only derivatives up to order $k$, we obtain a Banach space $\CC_k^p(G)$.
We write
$$
\CC_k^p(K\bs G/K)=\CC_k^p(G)\cap C(K\bs G/K)
$$
for the set of $K$-bi-invariant elements of $\CC_k^p(G)$.
\end{definition}

\begin{definition}
Let $L^1(G,\Xi)$ denote the $L^1$ space with respect to the measure $\Xi(x)dx$, where $dx$ denotes the Haar measure on $G$.
Further, let $L^1_0(G,\Xi)$ denote the space of $K$-bi-invariant functions in $L^1(G,\Xi)$.

Let $N\in\N$ and let $\CC_{N}(A_{sp})^W$ denote the space of all $N$-times continuously differentiable, $W$-invariant functions $g$ on $A_{sp}$ such that for every invariant differential operator $D$ of order $\le N$ one has 
$$
Dg(\exp(X))=O(\norm X^{-N}).
$$
for some fixed norm on $\a_{sp}$.
\end{definition}

\begin{proposition}\label{prop1.5}
\begin{enumerate}[\rm (a)]
\item The map $\Phi:f\mapsto f^{N_{sp}}(a)a^\rho$ is an isomorphism of topological algebras
$$
L^1_0(G,\Xi)\tto\cong L^1(A_{sp})^W.
$$
\item Let $N\in\N$ be large enough. Then $\CC_{N}(A_{sp})^W$ is a subset of $L^1(A_{sp})^W$ and  for every $g\in \CC_{N}(A_{sp})^W$ its inverse image $f=\Phi^{-1}(g)$ lies in $\CC_N^1(G)$.
\end{enumerate}
\end{proposition}

\begin{proof}
(a) For $f\in L^1_0(G,\Xi)$ we estimate the $L^1$-norm as
\begin{align*}
\norm{\Phi(f)}&=\int_{A_{sp}}|f^{N_{sp}}(a)|\,a^\rho\,da
=\int_{A_{sp}}\left|\int_{N_{sp}}f(an)\,dn\right|\,a^\rho\,da\\
&\le\int_{A_{sp}}\int_{N_{sp}}|f(an)|a^\rho\,dn\,da
=\int_{A_{sp}}\int_{N_{sp}}\int_K|f(ank)|a^\rho\,dk\,dn\,da\\
&=\int_G|f(x)|\ul a(x)^\rho\,dx
=\int_K\int_G|f(k^{-1}x)|\ul a(x)^\rho\,dx\,dk\\
&=\int_K\int_G|f(x)|\ul a(kx)^\rho\,dx\,dk=\int_G|f(x)|\Xi(x)\,dx.
\end{align*}
So $\Phi$ is a continuous homomorphism of Banach algebras. It is injective for the same reason as in Proposition \ref{prop1.2}. 
In the above estimate, note that we have equality throughout, if $f\ge 0$. The image of $\Phi$ contains $C_c^\infty(A_{sp})^W$ and therefore all monotonous limits of such functions, so it contains all positive functions in $L^1(A_{sp})^W$ and hence $\Phi$ is surjective.

(b)
As $f$ bi-invariant under $K$, in the expression $X*f*Y$ it suffices to consider $X,Y\in U(\a_{sp}\oplus\n_{sp})$.
The $\n$-derivatives are killed by the integral in $\Phi(f)(a)=\int_{N_{sp}}f(an)\,dn\,a^\rho$ and the $\a$-derivatives translate to derivatives of $g=\Phi(f)$.
This yields the claim.
\end{proof}

\subsection*{Orbital integrals}
 
Let $H$ be a $\theta$-stable Cartan subgroup of $G$, let
$\h_\C$ be its complex Lie algebra and let $\Phi =\Phi(\g_\C,\h_\C)$ be the
set of roots. Let $x\to x^c$ denote the complex conjugation on
$\g_\C$ with respect to the real form $\g$. Choose an
ordering $\Phi^+\subset \Phi$ and let $\Phi^+_I$ be the set of
positive imaginary roots. To any root $\alpha\in\Phi$ let
\begin{eqnarray*}
H &\to & \C^\times\\
 h &\mapsto & h^\alpha
\end{eqnarray*}
be its character, that is, for $X\in\g_\alpha$ the root space to
$\alpha$ and any $h\in H$ we have $\Ad(h)X=h^\alpha X$. Now put
$$
\Delta_I(h)= \prod_{\alpha\in\Phi_I^+}(1-h^{-\alpha}).
$$
Let $H=AT$ where $A$ is the connected split component and $T$ is
compact. An element $at\in AT=H$ is called {\it split
regular} if the centralizer of $a$ in $G$
equals the centralizer of $A$ in $G$. The split regular elements
form a dense open subset containing the regular elements of $H$.
Choose a parabolic $P$ with split component $A$, so $P$ has
Langlands decomposition $P=MAN$. For $at\in AT =H$ let
\begin{eqnarray*}
\Delta_+(at) &=& \left| \det((1-\Ad((at)^{-1}))|_{\g /\a\oplus
\m})\right|^{\frac1{2}}\\ {}\\
 & =& \left|\det((1-\Ad((at)^{-1}))|_\n)\right| a^{\rho_P}\\ {}\\
 &=&
 \left|\prod_{\alpha\in\Phi^+\sm\Phi_I^+}(1-(at)^{-\alpha})\right|
 a^{\rho_P},
\end{eqnarray*}
where $\rho_P$ is the half of the sum of the roots
in $\Phi(P,A)$, i.e., $a^{2\rho_P}=\det(a|\n)$. We will also write
$h^{\rho_P}$ instead of $a^{\rho_P}$.

For any $h\in H^{reg} = H\cap G^{reg}$ let
$$
F_f^H(h)= F_f(h)= \Delta_I(h) \Delta_+(h) \int_{G/A}
f(xhx^{-1}) dx.
$$
Then for $h\in
H^{reg}$ one has
$$
\CO_h(f) = \frac{F_f^H(h)}
                 {h^{\rho_P}\det(1-h^{-1}|(\g/\h)^+)},
$$
where $(\g/\h)^+$ is the sum of the root spaces attached to
positive roots. There is an extension of this identity to
non-regular elements as follows: For $h\in H$ let $G_h$ denote its
centralizer in $G$. Let $\Phi^+(\g_h,\h)$ be the set of positive
roots of $(\g_h,\h)$. Let
$$
\varpi_h = \prod_{\alpha\in\Phi^+(\g_h,\h)}Y_\alpha,
$$
then $\varpi_h$ defines a left invariant differential operator on
$G$.

\begin{lemma}\label{lem1.3}
For $f\in \CC_N^1(G)$ and $h\in H$ we have
$$
\CO_h(f) = \frac{\varpi_h F_f^H(h)}
                 {h^{\rho_P}\det(1-h^{-1}|(\g/\g_h)^+)}.
$$
\end{lemma}

\begin{proof}
This is proven in section 17 of \cite{HC-DS}. 
\end{proof}

Our aim is to express orbital integrals in terms of traces of
representations. By the above lemma it is enough to express
$F_f(h)$ it terms of traces of $f$ when $h\in H^{reg}$. For this
let $H_1=A_1T_1$ be another $\theta$-stable Cartan subgroup of $G$
and let $P_1=M_1A_1N_1$ be a parabolic with split component $A_1$.
Let $K_1=K\cap M_1$. Since $G$ is connected the compact group
$T_1$ is an abelian torus and its unitary dual $\widehat{T_1}$ is
a lattice. The Weyl group $W=W(M_1,T_1)$ acts on $\widehat{T_1}$
and $\widehat{t_1}\in\widehat{T_1}$ is called {\it regular} if its
stabilizer $W(\widehat{t_1})$ in $W$ is trivial. The regular set
$\widehat{T_1}^{reg}$ modulo the action of $W(K_1,T_1)\subset
W(M_1,T_1)$ parameterizes the discrete series representations of
$M_1$ (see \cite{Knapp}). For $\widehat{t_1}\in\widehat{T_1}$
Harish-Chandra \cite{HC-S} defined a distribution
$\Theta_{\widehat{t_1}}$ on $G$ which happens to be the trace of
the discrete series representation $\pi_{\widehat{t_1}}$ attached
to $\widehat{t_1}$ when $\widehat{t_1}$ is regular. When
$\widehat{t_1}$ is not regular the distribution
$\Theta_{\widehat{t_1}}$ can be expressed as a linear combination
of traces as follows. Choose an ordering of the roots of
$(M_1,T_1)$ and let $\Omega$ be the product of all positive imaginary roots.
For any $w\in W$ we have $w\Omega = \eps_I(w)\Omega$ for a
homomorphism $\eps_I : W\to \{ \pm 1\}$. For non-regular
$\widehat{t_1}\in\widehat{T_1}$ we get
$\Theta_{\widehat{t_1}}=\frac1{|W(\widehat{t_1})|}\sum_{w\in
W(\widehat{t_1})}\epsilon_I(w)\Theta'_{w,\widehat{t_1}}$, where
$\Theta'_{w,\widehat{t_1}}$ is the character of an irreducible
representation $\pi_{w,\widehat{t_1}}$ called a limit of discrete
series representation. We will write $\pi_{\widehat{t_1}}$ for the
virtual representation $\frac1{|W(\widehat{t_1})|}\sum_{w\in
W_{\widehat{t_1}}}\epsilon_I(w)\pi_{w,\widehat{t_1}}$.
Further let $\rho_I$ be $1/2$ times the sum of positive imaginary roots.

Let $\nu :a\mapsto a^\nu$ be a unitary character of $A_1$ then
$\widehat{h_1}=(\nu,\widehat{t_1})$ is a character of
$H_1=A_1T_1$. Let $\Theta_{\widehat{h_1}}$ be the character of the
representation $\pi_{\widehat{h_1}}$ induced parabolically from
$(\nu,\pi_{\widehat{t_1}})$. Harish-Chandra has proven

\begin{theorem}\label{inv-orb-int}
Let $H_1,\dots,H_r$ be a maximal a set of non-conjugate
$\theta$-stable Cartan subgroups with split components $A_1,\dots,A_r$. Let $H=H_j$ for some $j$ with split component $A$. Then
for each $j$ there exists a continuous function $\Phi_{H|H_j}$ on
$H^{reg}\times \hat{H_j}$ such that for $h\in H^{reg}$ it holds
$$
F_f^H(h)= \sum_{j=1}^r
\int_{\hat{H_j}}\Phi_{H|H_j}(h,\widehat{h_j})\
\tr\pi_{\widehat{h_j}}(f)\ d\widehat{h_j}.
$$
Further $\Phi_{H|H_j}=0$ unless there is $g\in G$ such that
$gAg^{-1}\subset A_j$. Finally for $H_j=H$ the function can be
given explicitly as
\begin{eqnarray*}
\Phi_{H|H}(h,\hat{h}) &=& \frac1{|W(G,H)|}\sum_{w\in
W(G,H)}\eps_I(w)\langle \hat{h},w^{-1}h\rangle h^{w\rho_I-\rho_I}
\end{eqnarray*}
where we recall that, although the modular shift $\rho$ need not be a character of $H$, the difference $w\rho_I-\rho_I$ always is.
\end{theorem}

\begin{proof}
\cite{HC-S}. 
\end{proof}

\subsection*{The split Cartan}
We continue to  assume that $G$ is a split group, i.e., it possesses a Cartan $H_{sp}=A_{sp} T_{sp}$, which is split and $\theta$-stable. Then $T_{sp}$ is the  centralizer of $A_{sp}$ in $K$,  it is a finite group. There are no imaginary roots, therefore $\eps_I$ is trivial and $\rho_I=0$.
We fix a choice of positive roots according to a parabolic $P_{sp}=A_{sp} M_{sp} N_{sp}$.
We write $A_{sp}^+$ for the positive Weyl chamber.
For $a\in A_{sp}^+$ and $t\in T_{sp}$ we have
$$
\CO_{at}(f)=\frac{F_f^{A_{sp} T_{sp}}(at)}{a^\rho\det(1-(at)^{-1}|\n_{sp})}.
$$
Further
\begin{align*}
&F_f^{A_{sp} T_{sp}}(at)\\
&=\sum_{\sigma\in\what T_{sp}}\int_{\a_{sp}^*}\Phi_{A_{sp}T_{sp}|A_{sp}T_{sp}}(at,(\sigma,i\la))\,\tr\pi_{\sigma,i\la}(f)\,d\la,\\
&=\frac1{|W|}\sum_{\sigma\in\what T_{sp}}\sum_{w\in W}\int_{\a_{sp}^*}a^{iw\la}\sigma(w^{-1}t)\,\tr\pi_{\sigma,i\la}(f)\,d\la,
\end{align*}
where we identify the real dual $\a_{sp}^*$ with $\what A_{sp}$ by sending $\la\in\a_{sp}^*$ to $a\mapsto a^{i\la}$ and the measure $d\la$ is the Plancherel measure.
Next we specialize to $f_0\in \CC_N^1(K\bs G/K)$, then $\pi_{\sigma,i\la}(f_0)=0$ unless $\sigma=1$, so that because of $\pi_{i\la}\cong\pi_{iw\la}$ we get
\begin{align*}
F_{f_0}^{A_{sp}T_{sp}}(at)
&=\frac1{|W|}\sum_{w\in W}\int_{\a_{sp}^*}a^{iw\la}\,\tr\pi_{\sigma,iw\la}(f)\,d\la\\
&=\int_{\a_{sp}^*}a^{i\la}\,\tr\pi_{\sigma,i\la}(f)\,d\la\\
&=\int_{\a_{sp}^*}h_{f_0}(i\la)\,e^{i\la(\log a)}\,d\la=a^\rho f_0^{N_{sp}}(a).
\end{align*}

We have shown

\begin{lemma}
If $A_{sp}T_{sp}$ is a split Cartan subgroup, then for $f_0\in \CC_N^1(K\bs G/K)$ and $a\in A_{sp}^+$, $t\in T_{sp}$ one has
$$
\CO_{at}(f_0)=\frac{f_0^{N_{sp}}(a)}{\det(1-(at)^{-1}|\n_{sp})},
$$
where $f_0^{N_{sp}}(a)=\int_{N_{sp}}f_0(an)\,dn$.
\end{lemma}

\section{Orbital integrals for SL(3)}
The group $G=\SL_3(\R)$ possesses two conjugacy classes of Cartan subgroups.
The split Cartan $A_{sp}T_{sp}$ and the fundamental Cartan $A_1T_1$.
Here $A_{sp}$ is the group of all diagonal matrices with positive entries and $T_{sp}$ is the group of all diagonal matrices with entries in $\{\pm 1\}$.
Further, $A_1$ is the group of all diagonal matrices of the form $\diag(y,y,y^{-2})$, $y>0$.
Finally, $T_1$ is the group of all matrices $\smat k\ \ 1$, $k\in \SO(2)$.

\subsection*{The fundamental Cartan}
Let $\SL_2^\pm(\R)$ denote the group of all real $2\times 2$ matrices of determinant $\pm 1$.
We fix the parabolic subgroup $P_1=M_1BN_1$, where $M_1$ is the group of all matrices of the form $\smat g\ \ {\det(g)}$ with $g\in \SL_2^{\pm}(\R)$ and $N_1$ is the group of all upper triangular matrices with ones on the diagonal and a zero in the (1,2)-position.
The Cartan $H_1=A_1T_1$ lies in $P_1$.
Let $\h_1$ be the Lie algebra of $H_1$.
Then for $f\in \CC_N^1(G)$ we have
\begin{align*}
F_f^{A_1T_1}(at)&=\int_{\what{A_{sp}T_{sp}}}\Phi_{A_1T_1|A_{sp}T_{sp}}(at,\hat h)\,\tr\pi_{\hat h}(f)\,d\hat h\\ 
&\quad
+\int_{\what{A_1T_1}}\Phi_{A_1T_1|A_1T_1}(at,\hat h)\,\tr\pi_{\hat h}(f)\,d\hat h.
\end{align*}
The second integral equals
\begin{align*}
\sum_{\chi\in\what T}\int_{\a_1^*}\Phi_{A_1T_1|A_1T_1}(at,(\chi,\mu))\,\tr\pi_{\sigma_{\chi},i\mu}(f)\,d\mu,
\end{align*}
where $\sigma_\chi\in \what M_1$ is the discrete series representation with Harish-Chandra parameter $\chi$.
As these do not contain $K\cap M_1$ invariant vectors, the operators $\pi_{\sigma_\chi,i\mu}(f_0)$ are all zero.
So in the case $f=f_0\in C(K\bs G/K)$ the second integral vanishes.
This means that for $f_0\in \CC_N^1(K\bs G/K)$ and $at\in A_1T_1$ we get
$$
F_{f_0}^{A_1T_1}(at)
=\int_{\a_{sp}^*}\Phi(at,i\la)
\,h_{f_0}(i\la)\,e^{i\la(\log a)}\,d\la.
$$
Where we have used the abbreviation $\Phi(at,i\la)=\Phi_{A_1T_1|A_{sp}T_{sp}}(at,i\la)$.

\begin{definition}
Let $\CC_{N}(A_{sp})^W_0$ denote the set of all $g\in \CC_{N}(A_{sp})^W$ which vanish on the walls of the Weyl chambers up to order $N$.
\end{definition}

\begin{lemma}
For given $m\in\N$ there exists  $N\in\N$ such that for every $g\in \CC_{N}(A_{sp})^W_0$, with $f=\Phi^{-1}(g)$ the map $t\mapsto \CO_{at}(f)$, defined for regular $t$, extends to an  $m$-times continuously differentiable function in $t$.
\end{lemma}

\begin{proof}
We need to compute the orbital integral of $at$ more explicitly.
Let again $P_1=M_1A_1N_1$ be the parabolic with split component $A_1$.
Let $\ga=at$ be a regular element of $A_1 T_1$.
Using the $M_1A_1N_1K$-integral formula, as $f_0$ is $K$-bi-invariant, we get
\begin{align*}
\CO_{at}(f_0)&=\int_{A_1T_1\bs G}f_0(x^{-1} at x)\,dx\\
&=\int_{T_1\bs M_1 N_1}f_0(n^{-1}m^{-1} at mn)\,dn\,dm\\
&= \frac1{\det(1-at|\n_1)}\int_{T_1\bs M_1}\int_{N_1}f_0(m^{-1}tm\, an)\,dn\,dm.
\end{align*}
Let $M_1^0\cong\SL_2(\R)$ be the connected component of $M_1$ and let $t_0$ be the matrix $\diag(-1,1,-1)$, so that $M_1=M_1^0\sqcup t_0M_1^0$.
We first use $K$-invariance of $f$ and then the $KAK$ integration formula on the group $M_1^0$ to get
\begin{align*}
\CO_{at}(f_0)&=\frac 2{\det(1-at|\n_1)}\int_{T_1\bs M_1^0}\int_{N_1}f_0(m^{-1}tm\, an)\,dn\,dm\\
&=\frac 2{\det(1-at|\n_1)}\int_{B^+}\int_{N_1}f_0(b^{-1}tb\, an)(b^\al-b^{-\al})\,dn\,db,
\end{align*}
where $B^+$ is the set of all matrices of the form $\diag(e^s,e^{-s},1)$ with $s>0$ and $\al$ is the positive root of $(B,M_1)$.
The function 
$$
f_{0,a}(x)=\int_{N_1}f_0(x\, an)\,dn
$$ 
is $K_1=K\cap M_1^0$ bi-invariant on $M_1^0$.
Hence there is a function $\phi_a$ of one variable such that
$$
f_{0,a}(x)=\phi_a(\tr(x^tx)-2).
$$
Writing $t=\diag(R_\theta,1)$ with $R_\theta=\smat{\cos\theta}{-\sin\theta}{\sin\theta}{\cos\theta}$ we get that 
$$
\CO_{at}(f_0)\det(1-at|\n_1)
$$ 
equals
\begin{align*}
&2\int_{B^+}f_{0,a}(b^{-1}tb)(b^\al-b^{-\al})\,db\\
&= 2\int_0^\infty f_{0,a}\(\smat{e^{-s}}\ \ {e^s} R_\theta\smat{e^s}\ \ {e^{-s}}\)(e^{2s}-e^{-2s})\,ds\\
&= 2\int_0^\infty f_{0,a}\smat{\cos\theta}{e^{-2s}\sin\theta}{e^{2s}\sin\theta}{\cos\theta}
(e^{2s}-e^{-2s})\,ds\\
&= 2\int_0^\infty \phi_{0,a}(2\cos^2+(e^{4s}+e^{-4s})\sin^2-2)
(e^{2s}-e^{-2s})\,ds\\
&= 2\int_0^\infty \phi_{0,a}(2\cos^2+(e^{2s}+e^{-2s})^2\sin^2-2\sin^2-2)
(e^{2s}-e^{-2s})\,ds\\
&= 2\int_0^\infty \phi_{0,a}(4\cos^2+(e^{2s}+e^{-2s})^2\sin^2-4)
(e^{2s}-e^{-2s})\,ds\\
 &= \int_0^\infty \phi_{0,a}(4\cos^2+(e^{s}+e^{-s})^2\sin^2-4)
(e^{s}-e^{-s})\,ds\\
 &= \int_2^\infty \phi_{0,a}(4\cos^2\theta+v^2\sin^2\theta-4)\,dv\\
  &= \int_2^\infty \phi_{0,a}\((v^2-4)\sin^2\theta\)\,dv=\int_0^\infty\frac{\phi_{0,a}(y)}{2\sqrt{4\sin^2\theta+y}\,|\sin\theta|}.
\end{align*}
This extends to an $m$-times differentiable function in the point $\theta=0$, if $\phi_{0,a}(y)$ vanishes to order $m$ at $y=0$.
Let $I$ denote the integral transform $I(h)(b)=\int_N h(bn)$, where $h$ is a function on $M_1^0\cong\SL_2(\R)$, $b=\diag(e^s,e^{-s})$ and $N$ is the group of all upper triangular matrices in $\SL_2(\R)$ with ones on the diagonal.
Then
\begin{align*}
I(f_{0,a})(b)&=\int_N\int_{N_1}f_0(bnan')\,dn\,dn'\\
&=\int_N\int_{N_1}f_0(abnn')\,dn\,dn'
=\int_{N_{sp}}f_0(abn)\,dn=g(ab)(ab)^{-\rho}.
\end{align*}
In terms of $\phi_{0,a}$ we have
\begin{align*}
I(f_{0,a})(b)&=I(f_{0,a})\smat{e^s}\ \ {e^{-s}}\\
&=\int_\R\phi_{0,a}\(\tr\smat{e^s}\ \ {e^{-s}}\smat 1x\ 1-2\)\,dx\\
&=\int_\R\phi_{0,a}\(e^{2s}+e^{-2s}+e^{2s}x^2-2\)\,dx\\
&=e^{-s}\CA(\phi_{0,a})\((e^s-e^{-s})^2\),
\end{align*}
where $\CA$ denotes the \e{Abel transform}
$$
\CA(\phi)(y)=\int_\R\phi(y+x^2)\,dx.
$$
By Lemma 11.2.7 of \cite{HA2}, the Abel transform is invertible, more precisely, one has the following:
Let $\phi$ be a continuously differentiable function on $[0,\infty)$ such that 
$$
|\phi(x+s^2)|,\, |s\phi'(x+s^2)|\le \al(s)
$$
for some $\al\in L^1([0,\infty))$ and all $x\ge 0$, then $q=\CA(\phi)$ is continuously differentiable and 
$$
\phi=\frac{-1}\pi \CA(q').
$$
This implies that vanishing up to some order of the function $g$ is inherited by the function $\phi_{0,a}$ and the lemma is proven.
\end{proof}

\subsection*{Pseudo-cusp forms}
\begin{definition}
A $K$-finite function $f\in L^1(G)$ is called a \e{trace class function}, if $\pi(f)$ is a trace class operator for every $\pi\in\what G$.

A trace class function is called a \e{pseudo-cusp form}, if $\tr\pi(f)=0$ for every $\pi\in\what G$ which is induced from the minimal parabolic $P_{sp}$.

If $\pi,\eta$ are representations of a group $G$ and $H$ is a subgroup, we write $[\pi:\eta]_H$ for
$$
\dim\Hom_H(V_\pi,V_\eta).
$$
\end{definition}

\begin{definition}
For $N\in\N$ let $\CC_N^\ev(\R)$ denote the space of all even $\phi\in C^N(\R)$ such that  
$$
\phi^{(m)}(x)=O(|x|^{-N})
$$
for every $0\le m\le N$.
Let $\phi\in\CC_N^\ev(\R)$.
Let $\tau\in\what K$ be given.
Then $s_0=\sup_{\pi\in\what G}\pi(C)<\infty$. 
By Proposition 2.1 of \cite{class} it follows that for sufficiently large $u\in\R$ there exists  a uniquely determined $f_{\phi,\tau,u}\in \CC_N^1(G)$ such that
$$
\pi\(f_{\phi,\tau,u}\)=\frac1{\dim\tau}\phi\(\sqrt{-\pi(C)-\frac14+u}\)\ P_{\pi,\tau}
$$
holds for every $\pi\in\what G$.
The dimension factor is there to give this operator the trace $\phi\(\sqrt{-\pi(C)-\frac14+u}\)[\pi:\tau]_K$, where
$$
[\pi:\tau]_K=\dim\Hom_K(V_\pi,V_\tau).
$$

\end{definition}

Let $\rho_1$ denote the half sum of positive roots of $(A_1,P_1)$. 
Recall that for every $k=0,1,2,\dots$ there exists an irreducible representation $\delta_{2k}$ of $K=\SO(3)$ of dimension $2k+1$ and this exhausts $\what K$.
For a virtual $K$-representation $\tau=\tau_+-\tau_-$ we define $f_{\phi,\tau,u}=f_{\phi,\tau_+,u}-f_{\phi,\tau_-,u}$.
For every $n\in\N$ there is a discrete series representation $\CD_{n+1}$ of $M_1=\SL_2^\pm(\R)$ which fits into an exact sequence
$$
0\to\CD_{n+1}\to \pi_{(-1)^{n+1},n}\to \sigma_{n-1}\to 0,
$$
where $\sigma_m$ is the representation of $\SL_2^\pm(\R)$ on the space of homogeneous polynomials $p(X,Y)$ of degree $m$.
The $\CD_{n+1}$ exhaust the discrete series of $M_1$.
Let $\g=\k\oplus\p$ be the Cartan decomposition of $\g$ with respect to $\k=\Lie(K)$.

For $k=0,1,2,\dots$ we consider the 
virtual representation
\begin{align*}
\tau_0&=\delta_4-\delta_2-2\delta_0,\\
\tau_k&=\tau_0\otimes\(\delta_{2k}-\delta_{2(k-1)}+\delta_{2(k-2)}-\dots\pm\delta_0\),\quad k\ge 1.
\end{align*}
For simplicity we  write $f_{\phi,k,u}$ instead of $f_{\phi,\tau_k,u}$.

\begin{lemma}
Let $\pi=\pi_{\CD_{n+1},\nu}$ be induced from the discrete series representation $\CD_{n+1}$ and $\nu\in\a_{1,\C}^*$.
\begin{enumerate}[\rm (i)]
\item Write $\p_{M_1}=\p\cap\m_1$, then $\tr\pi(f_{\phi,k,u})$ equals
$$
\tr\pi_{\CD_{n+1},\nu}(f_{\phi,k,u})=\begin{cases}\phi(\sqrt{|\nu|^2+u}),& n=k+1,\\
0&\text{otherwise.}\end{cases}
$$
\item If on the other hand, $\pi\in\what G$ is induced from $P_{sp}$, then $\tr\pi(f_{\phi,k,u})=0$, so $f_{\phi,k,u}$ is a pseudo-cusp form.
\end{enumerate}
\end{lemma}

\begin{proof}
For $t\in T_1\cong\SO(2)$ of the form $t=\smat a{-b}ba$ we write $\eps_m(t)=(a+bi)^m$.
Lemma 5.1 in \cite{class} tells us that in the case $k=0$ we have
$$
\tr\pi_{\CD_{n+1},\nu}(f_{\phi,0,u})
=\phi\(\sqrt{|\nu|^2+u}\)
\left[\CD_{n+1},\(\bigwedge^\odd\p_{M_1}-\bigwedge^\ev\p_{M_1}\)\right]_{K\cap M_1^+}.
$$
An inspection of the proof yields that the expression $\tr\pi_{\CD_{n+1},\nu}(f_{\phi,\tau_0\otimes\delta_{2k},u})$ equals $\phi\(\sqrt{|\nu|^2+a}\)$ times
$$
\left[\CD_{n+1},\(\bigwedge^\odd\p_{M_1}-\bigwedge^\ev\p_{M_1}\)\otimes\delta_{2k}\right]_{K\cap M_1^+}.
$$
The restriction to $\SO(2)$ equals
\begin{align*}
&\left.\(\bigwedge^\odd\p_{M_1}-\bigwedge^\ev\p_{M_1}\)\otimes\delta_{2k}\right|_{\SO(2)}\\
&=(\eps_2+\eps_{-2}-2\eps_0)(\eps_{k}+\eps_{k-1}+\dots+\eps_{-k})\\
&=\(\eps_{k+2}+\eps_{-(k+2)}\)
+\(\eps_{k+1}+\eps_{-(k+1)}\)
-\(\eps_{k}+\eps_{-k}\)
-\(\eps_{k-1}+\eps_{-(k-1)}\).
\end{align*}
We add these contributions up with alternate signs to get
\begin{align*}
&\tr\pi_{\CD_{n+1},\nu}(f_{\phi,k,u})\\
&=
\phi\(\sqrt{|\nu|^2+u}\)\,\left[\CD_{n+1}:\(\eps_{k+2}+\eps_{-(k+2)}\)-\(\eps_{k}-\eps_{-k}\)\right]_{K\cap M_1}.
\end{align*}
Now since
$$
\left.\CD_{n+1}\right|_{\SO(2)}=\bigoplus_{\substack{|k|\ge{n+1}\\ k\equiv (n+1)\mod 2}}\eps_k
$$
we find that there is only one $n$ with a non-zero contribution, and this is $n=k+1$.
This implies the first claim.
The second is obtained as in Lemma 5.1 of \cite{class}.
\end{proof}

We next compute orbital integrals for $f_{\phi,k,u}$.

\begin{lemma}\label{lem2.3} 
Let $\phi\in\CC_N^\ev(\R)$ and $a\in\R$ sufficiently large.
The function $F_{f_{\phi,k,u}}^{A_{sp}T_{sp}}$ vanishes identically. For $at\in A_1T_1$ we have $\Delta_I=(1-\eps_{-2})$ and for $k=0,1,2,\dots$ we have
$$
F_{f_{\phi,k,u}}^{A_1T_1}(at)=\Delta_I
\frac{\eps_{k+1}(t)+{\eps_{-k-1}(t)}}2\int_{\a_1^*}\phi\(\sqrt{|\nu|^2+u}\)\,a^{i\nu}\,d\nu.
$$
\end{lemma}

\begin{proof}
The first assertion is due to the fact that $f_{\phi,k,u}$ is a pseudo-cusp form.
It follows that
$$
F_{f_{\phi,k,u}}^{A_1T_1}(at)
=\frac12 \sum_{\chi\in\what{T_1}}\int_{\a_1^*}\(\sp{(\chi,i\nu),at}-\sp{(\ol{\chi},i\nu),at}\eps_{-2}(t)\)\,\tr\pi_{\tilde\chi,\nu}(f_{\phi,k,u})\,d\nu.
$$
Here $\tilde\chi$ stands for the (limit of) discrete series representation with Harish-Chandra parameter $\chi$.
For $\chi=\eps_0$, this is the Steinberg representation, on which $f_{\phi,k}$ has trace zero.
For $\chi=\eps_{\pm m}$ with $m\ge 1$ we have $\tilde\chi=\CD_{m+1}$
By the last lemma, only the summands with $|m|=k+1$ give non-zero contributions.
So the sum is $\int_{\a_1^*}\phi\(\sqrt{|\nu|^2+u}\)\,a^{i\nu}\,d\nu$ times
\begin{align*}
\eps_{k+1}-\eps_{-k-3}+\eps_{-k-1}-\eps_{k-1}=(1-\eps_{-2})(\eps_{k+1}+\eps_{-k-1}).\mqed
\end{align*}
\end{proof}

\section{The test function}
\subsection*{Symmetrized orbital integrals}
Let $f_0\in \CC_N^1(K\bs G/K)$ for $N$ large enough.
The function $F_{f_0}^{A_1T_1}(at)\Delta_I(t)^{-1}$ is even in $t\in T_1$, therefore it has a Fourier expansion
$$
\frac{F_{f_0}^{A_1T_1}(at)}{\Delta_I(t)}
=\tilde C(a)+\sum_{k=0}^\infty \(\eps_{k+1}(t)+\eps_{-k-1}(t)\) \tilde C_k(a)
$$
which is such that $C,C_k$ become arbitrarily smooth as $N$ increases and such that for every $m\in\N$ there is $N\in\N$ such that 
$$
|\tilde C_k(a)|\ll k^{-m}
$$
uniformly.
The function $\Delta_+$ is symmetric in $a$ in the sense that $\Delta_+(at)=\Delta_+(a^{-1}t)$ for $at\in A_1T_1$.
We make the invariant integral artificially symmetric by considering $SF^{A_1T_1}_{f_0}(at)=F_{f_0}^{A_1T_1}(at)+F_{f_0}^{A_1T_1}(a^{-1}t)$. We then have
$$
\frac{SF_{f_0}^{A_1T_1}(at)}{\Delta_I(t)}
=C(a)+\sum_{k=0}^\infty \(\eps_{k+1}(t)+\eps_{-k-1}(t)\) C_k(a),
$$
where $C, C_k$ are now symmetric, too.
For every $k\ge 0$ there exists $u_k\in\R$ and an even 
function $\phi_k\in\CC_N^\ev(R)$ such that $C_k(a)=\int_{\a_1^*}\phi_k\(\sqrt{|\nu|^2+u_k}\)\,a^{i\nu}\,d\nu$. 
Then the function $f_{\phi_k,k,u_k}$ does not depend on the choice of the pair $(\phi_k,u_k)$.
We therefore write it as
$$
f_{[k]}
$$
Let
$$
f_1=-\sum_{k=0}^\infty f_{[k]}.
$$
The sum converges uniformly with all derivatives up to order $N$. 

Let $S\CO_\ga(h)=\CO_\ga(h)+\CO_{\ga^{-1}}(h)$ denote the symmetrized orbital integral of a function $h$.

\begin{lemma}\label{lem3.1}
Suppose that $g\in \CC_{N}(A_{sp})^W$ vanishes on the walls of the Weyl chambers up to order one. Let $f_0=\Phi^{-1}(g)\in \CC_N^1(K\bs G/K)$, where $\Phi$ is the map of Proposition \ref{prop1.2}. Let $f_1$ be as above and set $f=f_0+f_1\in \CC_N^1(G)$. 
Then for any semisimple element $\ga\in G$ we have $S\CO_\ga(f)=0$ unless $\ga$ is  conjugate to an element $at\in A_{sp}^+ T_{sp}$, where $A_{sp}^+$ is the open positive Weyl chamber, in which case we call $\ga$ a \e{split regular element} and then we have
$$
S\CO_\ga(f)=\frac{g(a)+g(a^{-1})}{a^{\rho}\det(1-(at)^{-1}|\n_{sp})}
$$
\end{lemma}

\begin{proof}
We have $F_{f}^{A_{sp}T_{sp}}=F_{f_0}^{A_{sp}T_{sp}}$ since $f_1$ is a pseudo-cusp form.
By Lemma \ref{lem2.3} it follows that
$$
\frac{SF_{f}^{A_1T_1}(at)}{\Delta_I(t)}=C(a),
$$
for regular elements $at$ of $A_1T_1$, i.e., this symmetric invariant integral is independent of $t$.
Let $t(\theta)\in T_1$, where $\theta$ is the angle of the rotation, then $\Delta_I(at(\theta))=(1-e^{-2i\theta})$.
Let $\om=\om_a$, where $a\in A_1$ is regular, then by Lemma \ref{lem1.3} we have
\begin{align*}
S\CO_a(f)
&= \frac{\left.\om\( C(a)\Delta_I(t(\theta))\)\right|_{\theta=0}}{a^{\rho_1}\det(1-a^{-1}|(\g/\g_a)^+)}\\
&= \frac{2C(a)}{a^{\rho_1}\det(1-a^{-1}|\n_1)},
\end{align*}
where $P_1=A_1M_1N_1$ is the parabolic which contains $P_{sp}$ and has $A_1$ as a split component.
On the other hand, if $b=b_t=\diag(e^t,e^{-t},1)$ with $t\ne 0$, then $ab_t$ is a regular element of $A_{sp}$ and then $
\CO_{ab}(f)=\CO_{ab}(f_0)$ as $f_1$ is a pseudo-cups form.
Therefore, again by Lemma \ref{lem1.3} we get
$$
\CO_a(f)=\frac{\left.\frac d{dt}\right|_{t=0} f_0^{N_{sp}}(ab_t)}{a^{\rho_1}\det(1-a|\n_1)}
$$
It follows
$$
C(a)=\frac12 \left.\frac d{dt}\right|_{t=0} f_0^{N_{sp}}(ab_t)+f_0^{N_{sp}}(a^{-1}b_t).
$$
This vanishes, since $g$ vanishes on $a$ up to order $1$.
The claim follows.
\end{proof}

\subsection*{The twisting character}
The representation ring $\Rep(G)$ is freely generated by  two representations
\begin{align*}
\st:G&\to\GL_3(\C)&&\text{the standard representation},\\
\La&=\bigwedge^2\st&&\text{its exterior square}.
\end{align*}
Let
$$
\eta=\(\La+\st+2\)\otimes\(\La-\st\)
$$
as an element of $\Rep(G)$, i.e., a virtual representation.
Recall that the trace map $\pi\mapsto\tr\pi$ is a ring homomorphism from $\Rep(G)$ to $C(G)$, the algebra of continuous functions on $G$.

\begin{lemma}
For given $\pi\in\Rep(G)$, the trace $\pi(x)$ is a symmetric polynomial in the eigenvalues $u,v,w$ of $x$.
This constitutes a ring isomorphism
$$
\Rep(G)\cong \Z[u,v,w]^{\Per(3)}/(uvw-1).
$$
If $u,v,w$ are the complex eigenvalues of $x\in G$, then $uvw=1$ and
$$
\tr\eta(x)=(u^2-1)(v^2-1)(w^2-1).
$$
We have $\tr\eta(x^{-1})=-\tr\eta(x)$ and the ideal in $\Rep(G)$ of all $\pi$ with $\tr\pi$ vanishing on the set of parabolically singular elements is generated by $\eta$.
\end{lemma}

\begin{proof}
The first assertion is standard, the second is a matter of a computation involving $\st$ and $\La$ given by the symmetric polynomials $u+v+w$ and $uv+uw+vw$. 
The surprising fact that $\tr\eta(x)=-\tr\eta(x^{-1})$ uses $uvw=1$ several times in the computation
\begin{align*}
\tr\eta(x^{-1})
&=\(\frac1{u^2}-1\)\(\frac1{v^2}-1\)\(\frac1{w^2}-1\)\\
&=\(\frac{u^2v^2w^2}{u^2}-1\)\(\frac{u^2v^2w^2}{v^2}-1\)\(\frac{u^2v^2w^2}{w^2}-1\)\\
&=(u^2v^2-1)(u^2w^2-1)(v^2w^2-1)\\
&=1-(u^2+v^2+w^2)+(u^2v^2+u^2w^2+v^2w^2)-1\\
&= -(u^2-1)(v^2-1)(w^2-1)=-\tr\eta(x).
\end{align*}
Finally, the fact that $\eta$ generates the ideal follows from the fact that $x\in G$ is parabolically singular if and only if $x$ has $\pm 1$ for an eigenvalue.
\end{proof}

\subsection*{The geometric side}
Let $\al,\be$ denote the simple roots of $A_{sp}$ given by
$$
\al(\diag(x,y,z))=x-y,\quad\be(\diag(x,y,z))=y-z.
$$
We consider the set $A_{sp}^{++}$ of all $a\in A_{sp}^+$ such that $a^\al>a^\be$.
In coordinates, the set $A_{sp}^+$ is the set of diagonal matrices $\diag(u,v,w)$ with $u,v,w>0$ satisfying $uvw=1$ and $u>v>w$.
The set $A_{sp}^{++}$ is the subset of all $\diag(u,v,w)$ for which additionally $v<1$.
As $A_{sp}^+$ is a fundamental domain in $A_{sp}$ for the action of the Weyl group, the set $A_{sp}^{++}$ is a fundamental domain for the action of the group generated by the Weyl group and the transformation $a\mapsto a^{-1}$.
Likewise, let $A_{sp}^{+-}$ denote the set of all $a\in A_{sp}^+$ with $a^\al<a^\be$.

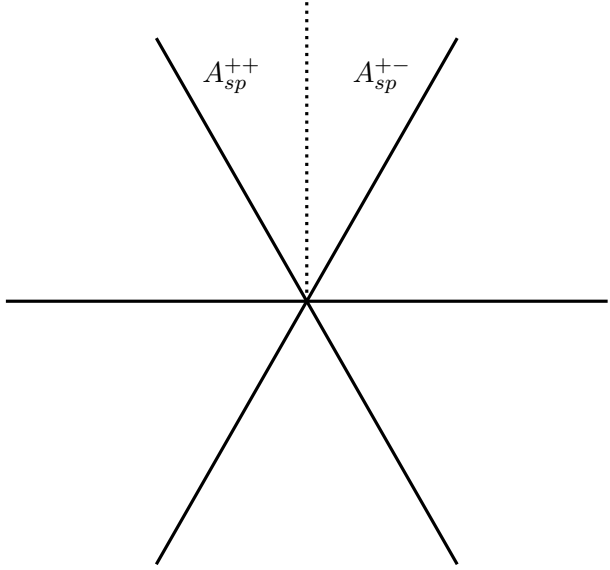
\begin{center}
\begin{tikzpicture}
\draw[very thick](-4,0)--(4,0);
\draw[very thick](-2,-3.5)--(2,3.5);
\draw[very thick](2,-3.5)--(-2,3.5);
\draw[very thick,dotted](0,0)--(0,4);
\draw(-1,3)node{$A_{sp}^{++}$};
\draw(1,3)node{$A_{sp}^{+-}$};
\end{tikzpicture}
\end{center}

\begin{proposition}
Let $g\in \CC_N(A_{sp})^W$ and suppose that $g$ vanishes on $A_{sp}^{+-}$.
Let $f$ be as in Lemma \ref{lem3.1}.
Let $\Ga$ be a congruence subgroup of $\SL_3(\Q)$ and let $K_\Ga$ be the closure of $\Ga$ in $\SL_3(\A_\fin)$, where $\A_\fin$ is the ring of finite $\Q$-adeles. Let $f_\fin$ be the characteristic function of $K_\Ga$ and let $f_\A=f_\fin\otimes f\tr\eta$ as a function on the adeles, where $\eta$ is the twisting character. Then the geometric side of the Arthur trace formula of $f$ equals
$$
J_\geom(f_\A)=\sum_{[\ga]}\vol(\Ga_\ga\bs G_\ga)
\frac{g(a_\ga)\tr\eta(\ga)}{a_\ga^{\rho}\det(1-(a_\ga t_\ga)^{-1}|\n_{sp})},
$$
where the sum runs over all conjugacy classes $[\ga]$ in $\Ga$ of split regular elements.
\end{proposition}

\begin{proof}
This summarizes the above computations and Corollary 1.3 of \cite{class}, except for one point. We have only computed the symmetrized orbital integrals instead of orbital integrals.
So $J_\geom(f_\A)$ equals
$$
\sum_{[\ga]}\vol(\Ga_\ga\bs G_\ga)
\CO_\ga(f)\tr\eta(\ga).
$$
As the orbital integrals vanish otherwise, we only have to consider those $\ga\in\Ga$ which are $G$-conjugate to some $a_\ga t_\ga\in A_{sp}^{++}M_{sp}$.
But then
\begin{align*}
S\CO_\ga(f)&=\frac{g(a_\ga)\tr\eta(\ga)+g(a_{\ga^{-1}})\tr\eta(\ga^{-1})}{a_\ga^{\rho}\det(1-(a_\ga t_\ga)^{-1}|\n_{sp})}=\frac{g(a_\ga)\tr\eta(\ga)}{a_\ga^{\rho}\det(1-(a_\ga t_\ga)^{-1}|\n_{sp})},
\end{align*}
since $g(a_{\ga^{-1}})=0$.
\end{proof}

\section{Prime Geodesic Theorem}
\subsection*{A test function}
For $s\in\C^2$ and $a\in A_{sp}^{++}$ write
$$
a^{-s}=a^{-s_1(\al-\be)-2s_2 \be}.
$$
The reason for this normalization will become transparent later.
For given $j\in\N$ and $s\in\C^2$ let $g_{j,s}\in C(A_{sp})^W$ be defined by
$$
g_{j,s}(a)=\begin{cases}\(\frac\partial{\partial s_1}
\frac\partial{\partial s_2}\)^{j+1} \ a^{-s},
&a\in A_{sp}^{++}\\
0&a\in A_{sp}^{+-}.
\end{cases}
$$

\begin{lemma}
Let $f_{j,s}=f_0+f_1$ be the function attached to $g=g_{j,s}$ as in Lemma \ref{lem3.1}.
For every $N\in\N$ there exists $C>0$ such that for every $s\in\C^2$ with $\Re(s_1),\Re(s_2)>C$ and every $j\ge N$ one has $f_{j,s}\in \CC_N^1(G)$.
\end{lemma}

\begin{proof}
This follows from Proposition \ref{prop1.5} (b).
\end{proof}

We now specify the multiple $B$ of the Killing form which fixes the Haar measures. We choose this in a way that the map $A\to\R^2$, $a\mapsto ((\al-\be)(\log a),\be(\log a))$ is measure-preserving.

\begin{lemma}
Let $\pi_{\mathrm{triv}}$ denote the trivial representation of $G$.
Write $D$ for the differential operator $\frac\partial{\partial s_1}\frac\partial{\partial s_2}$.
We have
$$
\tr(\pi_{\mathrm{triv}}(f_{j,s}\tr\eta))=\frac12D^{j+1}\(\frac1{(s_1-\frac73)(s_2-2)}+\frac1{(s_1-\frac53)(s_2-2)}\)+F(s),
$$
where $F$ is holomorphic in $\{\Re(s_1)>\frac73-\eps, \Re(s_2)>2-\eps\}$ for some $\eps>0$.
\end{lemma}

\begin{proof}
Note first that
$$
\tr\eta(a)=\(a^{\frac43(\al-\be)+2\be}-1\)
\(a^{-\frac23(\al-\be)}-1\)
\(a^{-\frac23(\al-\be)-2\be}-1\).
$$
As $a^{\al-\be}$ and $a^\be$ both tend to infinity, the leading term is $a^{\frac43(\al-\be)+2\be}$, followed by a summand $-a^{\frac23(\al-\be)+2\be}$.
All further summands have lower growth rate in both, $\al-\be$ and $\be$.
Using Weyl's integration formula, we compute
\begin{align*}
\tr\pi_{\mathrm{triv}}(f_{j,s}\tr\eta)
&=\int_Gf_{j,s}(x)\tr\eta(x)\,dx\\
&= \int_{A_{sp}^{++}}\frac14\sum_{t\in T_{sp}}\CO_{at}(f_{j,s})\tr\eta(at)\,|D_{sp}(at)|^2\,da\\
&\quad+\frac12\int_{H_1}\underbrace{\CO_h(f_{j,s})}_{=0}\tr\eta(h)\,|D_1(h)|^2\,dh.
\end{align*}
By our computation of the orbital integrals this equals
\begin{align*}
&\int_{A_{sp}^{++}}g_{j,s}(a)\frac14\sum_{t\in T_{sp}}a^\rho\det(1-(at)^{-1}|\n_{sp})\,\tr\eta(a)\,da\\
&=\int_{A_{sp}^{++}}g_{j,s}(a)(a^{\al+\be}-a^{-\al-\be})\,\tr\eta(a)\,da.
\end{align*}
So the leading term in $\tr\pi_{\mathrm{triv}}(f_{j,s}\tr\eta)$ is
\begin{align*}
&D^{j+1}\int_{A_{sp}^{++}}a^{-s_1(\al-\be)-2s_2\be}
a^{\frac73(\al-\be)+4\be}\,da\\
&= D^{j+1}\int_0^\infty\int_0^\infty e^{(\frac{7}3-s_1)x+2(2-s_2)y}\,dx\,dy
=\frac12D^{k+1}\frac1{(s_1-\frac73)(s_2-2)}.
\end{align*}
Taking the next term in the asymptotic of $\tr\eta(a)$ into account, gives the second term of the expansion in the lemma.
\end{proof}

\begin{lemma}
Let $J_\spec$ denote the spectral side of the trace formula.
There exists $\eps>0$ such that $J_\spec(\1_{K_\fin}\otimes f_{j,s}\tr\eta)-\tr\pi_\triv(\1_{K_\fin}\otimes f_{j,s}\tr\eta)$ converges locally uniformly absolutely for $s\in\Om_\eps$.
\end{lemma}

\begin{proof}
The spectral side $J_\spec(f)$ is a sum $\sum_\chi J_\chi(f)$ where $\chi$ runs through the set of conjugacy classes of pairs $(\CM_0,\pi_0)$ consisting of  a $\Q$-rational Levi subgroup $\CM$ and a cuspidal representation $\pi_0$ of $\CM_0$, the sum being absolutely convergent.
The particular terms have expansions
$$
J_\chi(f)=\sum_{\CM,\pi}J_{\chi,\CM,\pi}(f),
$$
running over all $\Q$-rational Levi subgroups $\CM$ containing a fixed minimal one and, for each $\CM$, over all discrete automorphic representations $\pi$ of $\CM$.
Explicitly,
$$
J_{\chi,\CM,\pi}(f)=\sum_{w\in W_\CM}c_{\CM,w}\int_{i\a_{\CL}^{*}}\sum_\CP\tr\(\mathfrak M_\CL(\CP,\nu)M(\CP,w)\rho_{\chi,\pi}(\CP,\nu,f)\)\,d\nu.
$$
Here, for a given element $w$ of the Weyl group of $\CM$, the Levi subgroup $\CL$ is determined by $\a_\CL=(\a_\CM)^w$, and $\CP$ runs through all parabolic subgroups having $\CM$ as Levi component.
The coefficient $c_{\CM,w}>0$ equals 1 in the case $\CM=\CG=\SL_3$.
We write $J_{\chi,\CM,\pi}^+(f)$ for the same expression with the trace replaced by the trace norm.

On $G$, our chosen multiple $B$ of the Killing form induces a left-invariant metric. The corresponding Laplace-Beltrami operator $\Delta\ge 0$ is given by $\Delta=-C+2C_K$, where $C$ is the Casimir of $G$ and $C_K$ is the Casimir of $K$ induced by the form $B$.
Then the operator $(\Delta+1)^{-m}$ has a kernel which is $2m-\dim G$ times continuously differentiable.
Further, the operator being left-invariant, it equals the convolution operator given by a function on $G$.  
In the proof of Theorem 3 of \cite{MuLa}, in particular formula (5.1), it is shown that
$J_{\spec}^+(\1_{K_\fin}\otimes(\Delta+1)^{-m})<\infty$ for $m$ sufficiently large.
Writing $f=(\Delta+1)^{-m}(\Delta+1)^{m}f$ and using the estimate $\norm{ST}_{\tr}\le\norm S_{\tr}\norm T_{\mathrm{op}}$, where $S$ and $T$ are operators and the norms are the trace norm and the operator norm, we see that it suffices to show that for $j$ and $m$ sufficiently large,  there exists an $\eps>0$ and a uniform bound on the operator norms of
$
\pi((\Delta+1)^{m}f_{j,s}\tr\eta)
$
for all $s\in\Om_\eps$ and all non-trivial $\pi\in\what G$. 
Now
\begin{align*}
\pi((\Delta+1)^mf_{j,s}\tr\eta)
&= \pi(\Delta+1)^m\pi(f_{j,s}\tr\eta)\\
&= \pi(\Delta+1)^m(1\otimes\tr)((\pi\otimes\eta)(f_{j,s})\\
&= (1\otimes\tr)\(\pi(\Delta+1)^m\otimes 1)((\pi\otimes\eta)(f_{j,s})\).
\end{align*}
We further have $\pi(\Delta+1)^m=(\pi(\Delta)+1)^m=(-\pi(C)+2\pi(C_K)+1)^m$.
Now $\pi(C)$ is a scalar and $\pi(C_K)$ is scalar on each $K$-type.
These scalars grow polynomially in $k$ for the enumeration $(\delta_k)$ of $K$-types.
But the sequence $f_{[k]}$ tends to zero faster than the any power of $k$. This reduces the proof to a single $K$-type, where the proof uses the explicit knowledge on $\pi(f_{j,s})$ for $\pi\in\what G$ and proceeds analogously to the proof of Lemma 8.3 in \cite{class}.
\end{proof}

\begin{definition}
Let $\CE(\Ga)$ denote the set of all conjugacy classes $[\ga]$ in $\Ga$ such that $\ga$ is in $G$ conjugate to an element $a_\ga$ in $A_{sp}^{++}$.
For $[\ga]\in\CE(\Ga)$ the element $\ga$ defines a closed geodesic in $\Ga\bs X$, where $X= G/K$ is the symmetric space attached to $G$. This closed geodesic lies in a unique 2-dimensional flat sub-manifold $F_\ga$ of $\Ga\bs X$. Let 
$$
\la_\ga=\text{ metric volume of }F_\ga.
$$
\end{definition}

\begin{theorem}[Prime Geodesic Theorem]
Let
$$
\La(T_1,T_2)=\sum_{\substack{[\ga]\in\CE(\Ga)\\
a_\ga^{\al-\be}\le T_1\\
a_\ga^{2\be}\le T_2}}\la_\ga.
$$
Then, as $T_1,T_2\to\infty$ independently we get
$$
\La(T_1,T_2)\sim T_1T_2.
$$
\end{theorem}

In the light of the Prime Geodesic Theorem for cocompact groups in \cite{GAFA} the coordinate $2\beta$ appears naturally.
In the present case we reduced the fundamental domain from $A_{sp}^+$ to $A_{sp}^{++}$ by taking inversion into account, this explains the coordinate $\al-\be$.

\begin{proof}
First define
$$
\ind(\ga)= \frac{\la_\ga}{\det(1-(a_\ga m_\ga)^{-1}|\n)}>0.
$$

For $a\in A_{sp}^{++}$ let
$$
l(a)=2(\al-\be)(\log a)\cdot\be(\log a).
$$
For $j\in\N$ consider the Dirichlet series
$$
L^j(s)=\sum_{[\ga]\in\CE(\Ga)}\ind(\ga)\, \tr\eta(\ga)\,l(a_\ga)^{j+1}a_\ga^{-s}a_\ga^{-\frac43(\al-\be)-2\be}.
$$
We have shown
$$
L^j(s)=D^{j+1}\(\frac1{(s_1-1)(s_2-1)}+\frac1{(s_1-\frac1{3})(s_2-1)}\)+ R(s),
$$
where $R(s)$ is holomorphic on $\{\Re(s_1),\Re(s_2)>1-\eps\}$ for some $\eps>0$.
Now recall
$$
\tr\eta(a)=\(a^{\frac43(\al-\be)+2\be}-1\)
\(a^{-\frac23(\al-\be)}-1\)
\(a^{-\frac23(\al-\be)-2\be}-1\).
$$
This implies that 
$$
\tr\eta(\ga)a_\ga^{-\frac43(\al-\be)-2\be}\to 1\quad \text{for}\quad a_\ga^{\al-\be},a_\ga^{2\be}\to\infty\quad \text{independently.}
$$ 
The higher-dimensional Wiener-Ikehara Theorem, i.e., Theorem 3.2 of \cite{GAFA}
implies that with
$$
\tilde\La(T_1,T_2)=\sum_{\substack{[\ga]\in\CE(\Ga)\\
a_\ga^{\al-\be}\le T_1\\
a_\ga^{2\be}\le T_2}}\ind(\ga)\,\tr\eta(\ga)a_\ga^{-\frac43(\al-\be)-2\be}.
$$
we get $\tilde\La(T_1,T_2)\sim T_1T_2$ as $T_1,T_2\to\infty$ independently.
Now
$$
\frac{\ind(\ga)\,\tr\eta(\ga)a_\ga^{-\frac43(\al-\be)-2\be}}{\la_\ga}
$$
tends to $1$ as $a_\ga^{\al-\be}$ and $a_\ga^{2\be}$ tend to $\infty$.
So the Prime Geodesic Theorem follows by Lemma 3.5 in \cite{GAFA}. 
\end{proof}

\subsection*{Application to class numbers}
In this section we give an asymptotic formula for class numbers of orders in number fields. It
is quite different from known results like Siegel's Theorem (\cite{ayoub}, Thm 6.2). The
asymptotic is in several variables and thus contains more information than a single variable one.
In a sense it states that the units of the orders are equally distributed in different directions if only
one averages over sufficiently many orders.

Let $O_\R(3)$ denote the set of all orders $\CO$ in totally real number fields $F$ of degree $3$. 
For such an order $\CO$ let $h(\CO)$ be
its class number,
$R(\CO)$ its regulator.

For $\la\in\CO^\times$ let $\rho_1,\rho_2,\rho_3$ denote the real embeddings of $F$ ordered
in a way that $|\rho_1(\la)|\ge |\rho_{2}(\la)|\ge |\rho_{3}(\la)|$. 
Let
$$
\al_1(\la)=\frac{|\rho_1(\la)\rho_3(\la)|}
{|\rho_{2}(\la)|^2},\quad \al_2(\la)=
\(\frac{|\rho_2(\la)|}
{|\rho_{3}(\la)|}\)^2
$$

\begin{theorem}\label{A.1}
For $T_1,\dots,T_r>0$ set
$$
\vartheta(T)=\sum_{\substack{\CO\in O_\R(3),\ \la\in\CO^\times/\pm
1\\
1<\al_1(\la)\le T_1\\
1<\al_2(\la)\le T_2\\ 
}}  R(\CO)\,
h(\CO).
$$
Then we have, as $T_1,T_2\to \infty$,
$$
\vartheta(T_1,T_2)\ \sim\ \frac{16}{\sqrt{3}}\,T_1T_2.
$$
\end{theorem}

\begin{proof}
For $\la\in\CE(\Ga)$, where $\Ga=\SL_3(\Z)$ the centralizer $F_\ga$ in $\M_3(\Q)$ is a totally real number field of degree 3 and all such occur in this way.
Next the centralizer $\CO_\ga$ of $\ga$ in $\M_3(\Z)$ os an order in $F_\ga$ and up to isomorphy, all totally real cubic orders occur,l each with order $\CO$ multiplicity $h(\CO)$, see \cite{classMP}.
The number $16/\sqrt 3$ occurs as renormalization factor between our measure on $A_{sp}$ and the measure used for the regulator.
\end{proof}

\section{A conjectural Lefschetz formula}
\subsection*{Spectral Lefschetz numbers}
Let $G$ denote a connected semisimple Lie group with finite center. Fix a
maximal compact subgroup $K$ with Cartan involution $\theta$. So $\theta$ is
an automorphism of $G$ with $\theta^2=\rm Id$ and $K$ is the set of all $x\in
G$ with $\theta(x)=x$.

Let $\g_\R, \k_\R$ denote the real Lie algebras of $G$ and $K$ and let $\g$
and $\k$ denote their complexifications. This will be a general rule: for a
Lie group $H$ we denote by $\h_\R$ the Lie algebra of $H$ and by
$\h=\h_\R\otimes\C$ its complexification.
Let
$b:\g\times\g\to\C$ be a positive multiple of the Killing form. On $G,K$ and
all parabolic subgroups as well as all Levi-components we install Haar measures
given by the form $b$ as in
\cite{HC-HA1}.

Let $H$ be a non-compact Cartan subgroup of $G$. Modulo conjugation we can
assume that
$H=AB$ where $A$ is a connected, $\theta$-stable  split torus and $B$ is a closed subgroup of
$K$. Fix a parabolic $P$ with split component $A$. Then $P$ has Langlands
decomposition $P=MAN$ and $B$ is a Cartan subgroup of $M$. Note that  an
arbitrary parabolic subgroup $P'=M'A'N'$ of $G$ occurs in this way if and
only if the group $M'$ has a compact Cartan subgroup. In this case we say
that $P'$ is a \emph{cuspidal parabolic}.

The choice of the parabolic $P$ amounts to the same as a choice of a set of
positive roots $\phi^+(\g,\a)$ in the root system $\phi(\g,\a)$. The Lie
algebra $\n$ of the unipotent radical $N$ can be described as
$\n=\bigoplus_{\al\in\phi^+(\g,\a)}\g_\al$, where $\g_\al$ is the \emph{root
space} attached to $\al$, i.e., $\g_\al$ is the space of all $X\in\g$ such
that $\ad(Y)X=\al(Y)X$ holds for every $Y\in\a$. Define
$\bar\n=\bigoplus_{\al\in\phi^+(\g,\a)}\g_{-\al}$. This is the 
\emph{opposite} Lie algebra. Let $\bar\n_\R=\bar\n\cap\g_\R$ and $\bar
N=\exp(\bar\n_\R)$. Then $\bar P= MA\bar N$ is the \emph{opposite parabolic}
to $P$.

Let $\a^*$ denote the dual space of $\a$. Since $A=\exp(\a_\R)$, every 
$\la\in\a^*$ induces a continuous homomorphism from $A$ to $\C^*$ written
$a\mapsto a^\la$ and given by $(\exp(H))^\la=e^{\la(H)}$.
Let $\rho_P\in\a^*$ be the half of the sum of all positive roots, each
weighted with its multiplicity. So $a^{2\rho_p}=\det(a\mid\n)$. Let
$\a_\R^-\subset\a_\R$ be the negative Weyl chamber consisting of all 
$X\in \a_\R$ such that $\al(X)<0$ for every $\al\in\phi^+(\g,\a)$. Let
$A^-=\exp(\a_\R^-)$ be the negative Weyl chamber in $A$. Further let
$\overline{A^-}$ be the closure of $A^-$ in $A$. This is a manifold with
corners. Let $K_M=M\cap K$. Then $K_M$ is a maximal compact subgroup of $M$
and it contains $B$. Fix an irreducible unitary representation
$(\sigma,V_\sigma)$ of $K_M$. Then $V_\sigma$ is finite dimensional. Let
$\breve\sigma$ be the dual representation to $\sigma$.

Let $\what G$ denote the unitary dual of $G$, i.e., it is the set of all
isomorphy classes of irreducible unitary representations of $G$. Let $\what
G_{\rm adm}\supset\what G$ be the admissible dual, i.e. the set of classes of admissible irreducible representations under infinitesimal equivalence. 
Harish-Chandra proved that two unitary irreducible representations are unitarily equivalent iff they are infinitesimally equivalent.
Therefore $\what G$ can be considered a subset of $\what G_{\rm adm}$.
For $\pi\in\what G_{\rm adm}$ let $\pi_K$ denote the $(\g,K)$-module of $K$-finite
vectors in $\pi$ and let $\La_\pi\in\h^*$ be a representative of the
infinitesimal character of $\pi$. Let
$H^\bullet(\n,\pi_K)$ be the Lie algebra cohomology with coefficients in
$\pi_K$. By \cite{Hecht-Schmid}, for each
$q\ge 0$ the
$(\a\oplus\m,K_M)$-module $H^q(\n,\pi_K)$ is admissible of finite length,
i.e., a Harish-Chandra module.

\begin{definition}
For $\la\in\a^*$ and an $A$-module $W$ let $W^\la$ be the generalized
$\la$-eigenspace, i.e., $W^\la$ is the set of all $w\in W$ such that there is
$n\in\N$ with 
$$
(a-a^\la)^n w= 0
$$
for every $a\in A$. 
Let $\m =\k_M\oplus\p_M$ be the Cartan decomposition of the Lie algebra $\m$
of $M$. For $\pi\in\what G$ and $\la\in\a^*$ let $L_\la^\sigma(\pi)$ denote the
\e{spectral Lefschetz number} given by
$$
L_\la^\sigma(\pi)= \sum_{p,q\ge 0}(-1)^{p+q+\dim
N}\dim\left(H^q(\n,\pi_K)^\la\otimes\wedge^p\p_M\otimes 
\breve\sigma\right)^{K_M}.
$$
\end{definition}

\begin{definition}
For a given smooth and compactly supported function $f\in
C_c^\infty(G)$ we define its Fourier transform $\hat f\colon \what G\to\C$ by
$$
\hat f(\pi):= \tr\pi(f).
$$
\end{definition}

\begin{proposition}
\begin{enumerate}[\rm(a)]
\item
For every $\phi\in C_c^\infty(A^-)$ there exists $f_\phi\in C_c^\infty(G)$ such
that for every $\pi\in\what G$,
$$
\widehat{f_\phi}(\pi)=\sum_{\la\in\a^*} L_\la^\sigma(\pi)\,\hat\phi(\la),
$$
where $\hat\phi$ is the Fourier transform of $\phi$, i.i.,
$\hat\phi(\la)=\int_A\phi(a)a^\la\,da$.
\item
The sum in {\rm (a)} is finite, more precisely,
the Lefschetz number $L_\la^\sigma(\pi)$ is zero unless there is an element $w$
of the Weyl group of $(\g,\h)$ such that
$$
\la= \left.\left( w\La_\pi\right)\right|_\a -\rho_P.
$$
\end{enumerate}
\end{proposition}

\begin{proof}
The proof of part (a) is contained in \cite{Lefschetz}, and
(b) is a consequence of Corollary 3.32 in \cite{Hecht-Schmid}.
\end{proof}

\subsection*{Geometric Lefschetz numbers}
Let $\Ga\subset G$ be a discrete subgroup of finite covolume.
Let $X= G/K$ be the symmetric space and $X_\Ga=\Ga\bs X=\Ga\bs G/K$ be the
corresponding locally symmetric quotient.

Let $\CE_P(\Ga)$ denote the set of
all conjugacy classes $[\ga]$ in $\Ga$ such that $\ga$ is in $G$ conjugate
to an element $a_\ga b_\ga$ of $A^- B$. Then there is a conjugate $H_\ga$ of
$H$ such that $\ga\in H_\ga$

For $[\ga]\in\CE_P(\Ga)$ we define the \emph{geometric Lefschetz number} by
$$
L^\sigma(\ga)= \vol(\Ga_\ga\bs G_\ga)\frac{\tr\sigma(b_\ga)}{\det(1-a_\ga b_\ga |\n)}.
$$

\subsection*{The Lefschetz formula}
The unitary $G$-representation on $L^2(\Ga\bs G)$ decomposes as
$$
L^2(\Ga\bs G)= L_{\rm disc}^2\oplus L_{\rm cont}^2,
$$
where 
$$
L_{\rm disc}^2= \bigoplus_{\pi\in\what G}N_\Ga(\pi)\,\pi
$$
is a direct sum of irreducibles with finite multiplicities and $L_{\rm
cont}^2$ is a sum of continuous Hilbert integrals. In particular, $L_{\rm
cont}^2$ does not contain any irreducible sub-representation.

Let $r$ be the dimension of $A$ and let $\al_1,\dots,\al_r\in \a_\R^*$ be
the primitive positive roots. Let $\a_\R^{*,+}=\{ t_1\al_1+\cdots +t_r\al_r
: t_1,\dots,t_r>0\}$ be the positive dual cone and let
$\overline{\a_\R^{*,+}}$ be its closure in $\a_\R^*$.

For $\mu\in\a^*$ and $j\in\N$ let $C^{\mu,j}(A^-)$ denote the space of all
functions on $A$ which
\begin{itemize}
\item
are $j$-times continuously differentiable on $A$,
\item
are zero outside $A^-$,
\item
satisfy $|D\phi|\le C |a^\mu|$ for every invariant differential operator $D$
on $A$ of degree $\le j$, where $C>0$ is a constant, which 
depends on $D$. 
\end{itemize}

This space can be topologized with the seminorms
$$
N_D(\phi)= \sup_{a\in A} |a^{-\mu}D\phi(a)|,
$$
$D\in U(\a)$, $\deg(D)\le j$. Since the space of operators $D$ as above is
finite dimensional, one can choose a basis $D_1,\dots,D_n$ and set
$$
\norm{\phi}= N_{D_1}(\phi)+\cdots +N_{D_n}(\phi).
$$
The topology of $C^{\mu,j}(A^-)$ is given by this norm and thus
$C^{\mu,j}(A^-)$ is a Banach space.

\begin{conjecture}\label{Lefschetz}
(Lefschetz Formula)\\
For $\la\in \a*$ and $\pi\in\what
G_{\rm adm}$ there is an integer $N_{\Ga,\rm cont}(\pi,\sigma,\la)$ which
vanishes if $\Re(
\la)\notin \overline{\a_{\R}^{*,+}}$ and there are $\mu\in\a^*$ and
$j\in\N$ such that for each
$\phi\in C^{\mu,j}(A^-)$ and with
$$
\tilde N_{\Ga}(\pi,\sigma,\la)= N_\Ga(\pi)+N_{\Ga,\rm cont}(\pi,\sigma,\la)
$$
we have
$$
\sum_{\stackrel{\pi\in\what G}{\la\in\a^*}} 
\tilde N_\Ga(\pi,\sigma,\la)\,
L_\la^\sigma(\pi)\,\int_A\phi(a) a^\la\, da
=
 \sum_{[\ga]\in\CE_P(\Ga)} L^\sigma(\ga)\,\phi(a_\ga).
$$
Either side of this identity represents a continuous functional on
$C^{\mu,j}(A^-)$.
\end{conjecture}

In the following cases the conjecture is known.
\begin{enumerate}
\item
The conjecture holds if $\Ga$ is cocompact.
In that case the numbers $N_{\Ga,\rm cont}(\pi,\sigma,\la)$ are all zero. This is
shown in \cite{Lefschetz}.
\item
In the next sections we will prove the conjecture for $G=\PGL_2(\R)$ and $\PGL_3(\R)=\SL_3(\R)$.
\end{enumerate}

\section{PGL(2)}
Let $G=\PGL_2(\R)=\GL_2(\R)/\R^\times=\SL_2^\pm(\R)/\pm 1$ and for convenience write 
$\sbmat abcd$ for the element $\R^\times\smat abcd$ of $G$.
We further write
$K=\PO(2)=\mathrm O(2)/\pm 1$ as well as 
$$
A=\left\{\bmat {e^t}\ \ {e^{-t}}:t\in\R\right\},\qquad N=\left\{\bmat 1x\ 1:x\in\R\right\}.
$$
Then $P=AN$ is a minimal parabolic subgroup of $G$.
We write $\a_\R$ for the real Lie algebra of $A$ and $\a$ for its complexification.
The dual space $\a^*$ can be identified with the set of continuous group homomorphisms $A\to\C^*$, we write $a\mapsto a^\phi$ for the element given by $\phi\in\a^*$.
Let $\rho\in\a^*$ be the modular shift, i.e., $\rho\smat{t}\ \ {-t}=t$.
We identify $\a^*$ with $\C$ by sending $\rho$ to $1/2$.
So any $\phi\in\a^*$ will be written as $\phi=2\la\rho$ with $\la\in\C$.

Let $\what G_\adm$ denote the set of all irreducible admissible continuous Banach representations of $G$ up to infinitesimal equivalence, or, what amounts to the same, the set of all irreducible admissible $(\g,K)$-modules up to isomorphy, where $\g$ is the complexified Lie algebra of $G$, see \cite{Wallach}.
We will describe the set $\what G_\adm$ below.
First we give some representations.

\begin{itemize}
\item The finite-dimensional representations are $\delta_{2k}$ and $\det\otimes\delta_{2k}$,  for an integer $k\ge 0$, where $\det$ stands for the determinant on $\SL_2^\pm(\R)$ and $\delta_{2k}$, $k\in\N_0$ is the $2k+1$-dimensional representation on the space of $G$ on the space of all polynomials $p(X,Y)$ in two variables which are homogeneous of degree $2k$.
A basis of this space is $X^{2k},X^{2k-1}Y,\dots,Y^{2k}$. The representation is $\delta_{2k}(g)p(X,Y)=p((X,Y)g)$.
\item Let $M$ be the subgroup of $G$ consisting of two elements, the non-trivial being $T=\sbmat {-1}\ \ 1$. 
Let $\sigma\in\what M$ be a character.
For $\la\in\C$ let $\pi_{\sigma,2\la\rho}$ be the representation on the space of all  functions $f:G\to\C$ with $f(manx)=\sigma(m)a^{2\la\rho+\rho}f(x)$, such that $f|_K\in L^2(K)$.
The representation is given by $\pi_{\sigma,2\la\rho}(y)f(x)=f(xy)$.

\end{itemize}

The set $\what G_\adm$ consists of
\begin{enumerate}[\rm(a)]
\item The finite-dimensional representations $\delta_{2k}$, $\det\otimes\delta_{2k}$ with $k\ge 0$ an integer.
The dimension of $\delta_{2k}$ or $\det\otimes\delta_{2k}$ is $2k+1$.
\item The induced representations $\pi_{\sigma,2\la\rho}$ with $\sigma\in\what M$, $\la\in\C$, $\la\notin \frac12+\Z$.
\item 
The discrete series representations $\CD_{n+1}$ for $n\in\N$ odd.
\end{enumerate}
The only isomorphisms between these are 
$$
\pi_{\sigma,2\la\rho}\cong\pi_{\sigma,-2\la\rho},\qquad \la\notin \frac12+\Z.
$$
For details see \cite{Knapp}. For a given odd number $n\in\N$ there are exact sequences
$$
0\to \CD_{n+1}\to\pi_{\sigma,n\rho}\to\sigma\otimes\delta_{n-1}\to 0,
$$
where by abuse of notation we have written 
$$
\sigma\otimes\delta_{2k}=\begin{cases}
\delta_{2k}&\sigma=1,\\
\det\otimes\delta_{2k}& \sigma\ne 1.\end{cases}
$$
 and
$$
0\to\sigma\otimes\delta_{n-1}\to\pi_{\sigma,-n\rho}\to\CD_{n+1}\to 0.
$$
These $(\g,K)$-morphisms are unique up to multiplication by scalars.

\begin{definition}
or $p\ge 0$, let $\H^p(\n,\pi)$ denote the Lie algebra cohomology  of $\n=\Lie_\C(N)$ with coefficients in $\pi\in\what G_\adm$, see \cite{Borel-Wallach}.
Note that, as $\dim\n=1$, the space  $\H^p(\n,\pi)$ can be non-zero only if $p= 0$ or $p=1$.
\end{definition}

\begin{definition}
We write $\sigma\otimes\C_{2\la\rho}$ for the one-dimensional $\C[AM]$-module on which $A$ acts via the morphism $2\la\rho$ and $M$ acts via $\sigma$.
As an abbreviation, we will sometimes write $\C_{2\la\rho}$ for $1\otimes\C_{2\la\rho}$.
\end{definition}

\begin{proposition}\label{prop1.3}
As $AM$-module, the Lie algebra cohomology $\H^p(\n,\pi)$ for $\pi\in\what G_\adm$ is isomorphic to the following:

\begin{itemize}
\item $\H^0(\n,\pi_{\sigma,2\la\rho})=0$ for $\la\in\C$ except for $\la=-n/2$ for odd $n\in\N$.\\ 
$\H^1(\n,\pi_{\sigma,2\la\rho})\cong (\sigma\det)\otimes\C_{2\la\rho-\rho}\oplus(\sigma\det)\otimes\C_{-2\la\rho-\rho}$, if $\la\notin \frac12+\Z$.

\item $\H^0(\n,\sigma\otimes\delta_{2k})\cong(\sigma\det)\otimes\C_{2k\rho}$,\\
$\H^1(\n,\sigma\otimes\delta_{2k})\cong (\sigma\det)\otimes\C_{-2k\rho-2\rho}$.

\item $\H^0(\n,\CD_{n+1})=0$,\ \ \ \ $\H^1(\n,\CD_{n+1})\cong \(\det\otimes\C_{n\rho-\rho}\)\oplus \C_{n\rho-\rho}$.
\end{itemize}
\end{proposition}

\begin{proof}
Let $f\in \H^0(\n,\pi_{\sigma,2\la\rho})$. Then $f$ is a smooth function $f:G\to\C$ satisfying $f(manx)=\sigma(m)a^{2\la\rho+\rho}f(x)$, $x\in G$, as well as $f(x\sbmat 1u\ 1)=f(x)$ for every $u\in\R$. By means of the Iwasawa decomposition we compute
$$
f\(\sbmat \ {-1}1\ \)=f\(\sbmat \ {-1}1\ \sbmat 1u\ 1\)
= f\(\frac1{\sqrt{1+u^2}}\sbmat u{-1}1u\)(1+u^2)^{-\la-\frac{1}2}.
$$
If $\Re(\la)>-\frac12$, then the right hand side tends to zero as $u\to\infty$. This gives $f\(\sbmat \ {-1}1\ \)=0$, which, again by the same equation, implies  $f=0$.
The first claim therefore is proven for $\Re(\la)>-\frac12$.
As $\pi_{\sigma,2\la\rho}\cong\pi_{\sigma,-2\la\rho}$ if $\la\notin\frac12+\Z$, the claim follows in general.

For $\H^1(\n,\pi_{\la\rho})$ we firstly use Poincar\'e duality,
$$
\H_p(\n,\pi)\cong\H^{1-p}(\n,\pi)\otimes\det(\n)\cong\H^{1-p}(\n,\pi)\otimes\(\det\otimes\C_{2\rho}\),
$$
and secondly  Frobenius reciprocity
$$
\Hom_{AM}(\H_0(\n,\pi),\sigma\otimes\C_{2\la\rho+\rho})
\cong \Hom_G(\pi,\pi_{\sigma,2\la\rho}).
$$
These two are standard and can be found in \cite{Borel-Wallach} or in \cite{Hecht-Schmid}, for instance.
By the classification of $\what G_\adm$ we have $\Hom_G(\pi_{\sigma,2\la\rho},\pi)\ne 0$ if and only if $\pi=\pi_{\sigma,\pm2\la\rho}$ and in this case we derive from Frobenius reciprocity that
$$
\H_0(\n,\pi_{\sigma,2\la\rho})\cong \(\sigma\otimes\C_{2\la\rho+\rho}\)\oplus\(\sigma\otimes\C_{-2\la\rho+\rho}\).
$$
By Poincar\'e duality we get the claim.

The representation $\delta_{2k}$ can be realized as representation on the space of all polynomials $p(X,Y)$ in two variables which are homogeneous of degree $2k$.
A basis of this space is $X^{2k},X^{2k-1}Y,\dots,Y^{2k}$. The representation is $\delta_{2k}(g)p(X,Y)=p((X,Y)g)$.
It emerges that $\H^0(\n,\sigma\otimes\delta_{2k})$ is spanned by $1\otimes X^{2k}$ and since $AM$ acts on this by $\sigma\otimes 2 k\rho$, we get $\H^0(\n,\sigma\otimes\delta_{2k})\cong\sigma\otimes\C_{2k\rho}$.
Further, as $\sigma\otimes\delta_{2k}$ is a sub-representation of $\pi_{\sigma,-2k\rho-\rho}$,  Frobenius reciprocity yields $\H_0(\n,\sigma\otimes\delta_{2k})\cong \sigma\otimes\C_{-2k\rho}$ so that again Poincar\'e duality  gives the claim.

As $\CD_{n+1}$ is a sub-representation of $\pi_{\sigma,n\rho}$ the fact that $\H^0(\n,\CD_{n+1})=0$ follows from $\H^0(\n,\pi_{\sigma,n\rho})=0$. Finally, $\CD_{n+1}$ is a sub-representation of $\pi_{\sigma,n\rho}$ in essentially one way, so Frobenius reciprocity yields $\H_0(\n,\CD_{n+1})\cong\C_{n\rho+\rho}\oplus\(\det\otimes\C_{n\rho+\rho}\)$ and the claim follows by duality.
\end{proof}

\begin{definition}
For a complex vector space $V$ on which $AM$ acts linearly, let $V^{\sigma}_{2\la\rho}$ be the space of all $v\in V$ with $\(am-\sigma(m)a^{2\la\rho}\)^nv=0$ for some $n\in\N$ and all $am\in AM$. 

For $\sigma\in\what M$, $\la\in\a^*$ and $\pi\in\what G_\adm$ define the \e{spectral Lefschetz numbers} by
$$
L_{2\la\rho}^\sigma(\pi)=\sum_{q=0}^{\dim \n}(-1)^{q+\dim\n}\dim\H^q(\n,\pi)^{\sigma}_{2\la\rho}
$$
\end{definition}

\begin{corollary}
The computations in Proposition \ref{prop1.3} yield
\begin{align*}
L_{2\la\rho}^\sigma(1)&=\begin{cases}
1&\la=-1,\\
-1&\sigma=\det,\ \la=0,\\
0&\text{otherwise.}\end{cases}\\
L_{2\la\rho}^\sigma(\det)&=\begin{cases}
1&\sigma=1,\ \la=-1,\\
-1&\sigma=1,\ \la=0,\\
0&\text{otherwise.}\end{cases}\\
L_{2\la\rho}^\sigma(\pi_{\sigma_1,{2\mu\rho}})&=\begin{cases}
1&\sigma=\sigma_1\det,\ \la=\pm\mu-\frac12\\
0&\text{otherwise.}\end{cases}\\
L_{2\la\rho}^\sigma(\CD_{n+1})&=\begin{cases}
1&\la=\frac{n-1}2\\
0&\text{otherwise.}\end{cases}\\
\end{align*}
\end{corollary}

Next for the geometric Lefschetz numbers. If $[\ga]\in\CE_P(\Ga)$, then $\ga$ is
$G$-conjugate to $\sbmat {N(\ga)^{1/2}}{}{}{N(\ga)^{-1/2}}$ for some
$N(\ga)>1$. An element $\ga$ of $\Ga$ is called \emph{primitive} if
$\ga=\sigma^n$ for
$\sigma\in\Ga$ and $n\in\N$ implies $n=1$. Each $\ga\in\CE_P(\Ga)$ is a
power of a unique primitive $\ga_0$ which will be called the primitive
\emph{underlying} $\ga$.
Then
$$
L(\ga)=\frac{\log N(\ga_0)}{1-N(\ga)^{-1}}.
$$

We will now recall some facts about the \emph{Selberg zeta function}
\cites{Efrat, Hejhal2, Iwaniec}. Let $\CE_P^p(\Ga)$ denote the set of all
primitive classes in $\CE_P(\Ga)$. The Selberg zeta function is given by
the product
$$
Z(s)=\prod_{\ga\in\CE_P^p(\Ga)}\prod_{k=0}^\infty \left(
1-N(\ga)^{-s-k}\right).
$$
The product converges locally uniformly for $\Re(s)>1$. The zeta function
extends to a meromorphic function on the plane of finite order. It has a
simple zero at $s=1$ and zeros at $s=\frac 12 \pm u$ of multiplicity
$N_\Ga(\pi_{u\frac\rho 2})$. These are all zeros or poles in $\Re(s)\ge
\frac 12$ except for $s=\frac 12$ where $Z(s)$ has a zero or pole of order
$N_\Ga(\pi_0)$ minus the number of cusps. The poles and zeros in
$\Re(s)<\frac 12$ can be described through the scattering matrix or
intertwining operators \cites{Efrat, Hejhal2, Iwaniec}.

Recall the inversion formula for the Mellin transform. Let the
function $\psi$ be integrable on $(0,\infty)$ with respect to the
measure $\frac{dt}t$, in other words, $\psi\in L^1\left((0,\infty),
\frac{dt}t\right)$. Then the \emph{Mellin transform} of $\psi$ is given by
$$
M\psi(s)= \int_0^\infty t^s\,\psi(t)\,\frac{dt}t,\qquad s\in i\R.
$$
If the function $\psi$ is continuously differentiable and $\psi'(t)t, \psi''(t)t^2$ are
also in $L^1\left((0,\infty),\frac{dt}t\right)$, then the following
inversion formula holds,
$$
\psi(t)= \frac{1}{2\pi i} \int_{i\R} M\psi(s) t^{-s}\, ds.
$$
Now assume that $\psi$ is supported in the interval $[1,\infty)$ and that
for some $\mu>0$ the functions $\psi(t),\psi'(t)t,\psi''(t)t^2$ all are
$O(t^{-\mu})$. Then it follows that the integral
$M\psi(s)$ defines a function holomorphic in
$\Re(s)<\mu$ and the integral in the
inversion formula can be shifted,
$$
\psi(t)= \frac{1}{2\pi i} \int_{C-i\infty}^{C+i\infty} M\psi(s)\, t^{-s}\,
ds,
$$
for every $C<\mu$.

Every $\ga\in\CE_P(\Ga)$ can be
written as
$\ga=\ga_0^n$ for some uniquely determined $\ga_0\in\CE_P^p(\Ga)$ and a
unique $n\in\N$.
A computation yields for $\Re(s)>1$,
\begin{eqnarray*}
\frac {Z'}Z(s)&=& \sum_{\ga\in\CE_P^p(\Ga)}\sum_{n=1}^\infty \frac{\log
N(\ga)}{1-N(\ga)^{-n}}\, N(\ga)^{-ns}\\
&=& \sum_{\ga\in\CE_P(\Ga)} \frac{\log
N(\ga_0)}{1-N(\ga)^{-1}}\, N(\ga)^{-s}
\end{eqnarray*}
Let $\psi$ be as above with $\mu>1$ and let $1<C<\mu$. Then, since
$\frac{Z'}Z(s)$ is bounded in $\Re(s)=C$ we can interchange integration and
summation to get
\begin{eqnarray*}
\frac 1{2\pi i}\int_{C-i\infty}^{C+i\infty} \frac{Z'}Z(s)\, M\psi(s)\, ds
&=& \sum_{\ga\in\CE_P(\Ga)}\frac{\log N(\ga_0)}{1-N(\ga)^{-1}}\psi(N(\ga)).
\end{eqnarray*}\vspace{10pt}

For $a\in A^-=\left\{\sbmat t{}{}{t^{-1}} : 0<t<1\right\}$ set
$$
\phi(a)= \phi\mat t{}{}{t^{-1}}= \psi\left(\frac 1t\right).
$$
Then $\phi\in C^{2,2\mu\rho}(A^-)$ and 
\begin{eqnarray*}
\frac 1{2\pi i}\int_{C-i\infty}^{C+i\infty}\frac{Z'}Z(s)\, M\psi(s)\, ds &=&
\sum_{\ga\in\CE_P(\Ga)}\frac{\log N(\ga_0)}{1-N(\ga)^{-1}}\phi(a_\ga)\\
&=& \sum_{\ga\in\CE_P(\Ga)}L(\ga)\phi(a_\ga),
\end{eqnarray*}
which is the right hand side of the Lefschetz formula.

Now suppose that $\phi\in C^{j,2\mu\rho}(A^-)$ for some $j\in\N$ and some
$\mu>1$. Then the functions $\psi(t),\psi'(t)t,\dots,\psi^{(j)}(t)t^j$ are
all $O(t^{-\mu})$. Integration by parts shows that
$$
M\psi(t)\-\frac{(-1)^j}{s(s+1)\cdots (s+j-1)}\int_0^\infty
t^{s}\,\psi^{(j)}(t)t^j\,\frac {dt}t.
$$
This implies that $M\psi(s)=O\left((1+|s|)^{-j}\right)$ uniformly in
$\{\Re(s)\le\al\}$ for every $\al <\mu$.

For $R>0$ and $a\in\C$ let $B_r(a)$ be the closed disk around $a$ of radius
$r$. Let $g$ be a meromorphic function on $\C$ with poles $a_1, a_2,\dots$.
We say that $g$ is \emph{essentially of moderate growth}, if there is a
natural number $N$, a constant $C>0$, and  as sequence of real numbers
$r_n>0$ tending to zero, such that the disks $B_{r_n}(A_n)$ are pairwise
disjoint and that on the domain $D=\C\setminus\bigcup_n B_{r_n}(a_n)$ it
holds $|g(z)|\le C|z|^N$. Every such $N$ is called a 
\emph{growth exponent}
of $g$.

\begin{lemma}
Let $h$ be a meromorphic function on $\C$ of finite order and let $g=h'/h$
be its logarithmic derivative. Then $g$ is essentially of moderate growth
with growth exponent equal to the order of $h$ plus two.
\end{lemma}

\begin{proof}
This is a direct consequence of Hadamard's factorization Theorem applied to
$h$.
\end{proof}

This Lemma together with the growth estimate for $M\psi$ implies that for
$j$ large enough the contour integral over $C+i\R$ can be moved to the left,
deforming it slightly, so that one stays in the domain $D$, and gathering
residues. Ultimately, the contour integral will tend to zero, leaving only
the residues. One gets
\begin{eqnarray*}
\sum_{\ga\in \CE_P(\Ga)} L(\ga)\, \phi(a_\ga) &=& \sum_{s_0\in\C}\left(
\res_{s=s_0}\frac{Z'}Z(s)\right)\, M\psi(s_0)\\
&=& \sum_{s_0\in\C}\left(
\res_{s=s_0}\frac{Z'}Z(s)\right)\, \int_0^\infty\psi(t) t^{s_0}\frac{dt}t\\
&=&\sum_{s_0\in\C}\left(
\res_{s=s_0}\frac{Z'}Z(s)\right)\, \int_{A^-}\phi(a) a^{-s_0\rho}\, da.
\end{eqnarray*}
This implies the conjecture in the case $G=\PGL_2(\R)$.

{\small Mathematisches Institut\\
Auf der Morgenstelle 10\\
72076 T\"ubingen\\
Germany\\
\tt deitmar@uni-tuebingen.de}

\today

\end{document}